\documentclass[openany, a4paper,11pt]{article}
\usepackage{amsfonts}
\usepackage{amsmath}\usepackage{mathrsfs,amsfonts}

\usepackage{amsmath,amssymb,amsfonts}
\usepackage{amssymb}
\usepackage{color}
\definecolor{c}{rgb}{0.9,0.3,0.1}
\definecolor{b}{rgb}{0.1,0.3,0.9}

\allowdisplaybreaks[4]
 \makeatletter
 \@addtoreset{equation}{section}

\newcommand{\qed}{\hphantom{.}\hfill $\Box$\medbreak}
\newtheorem{theorem}{Theorem}[section]

\newtheorem{lemma}[theorem]{Lemma}

\newtheorem{proposition}[theorem]{Proposition}
\newtheorem{corollary}[theorem]{Corollary}
\newtheorem{remark}[theorem]{Remark}

\newfam\msbmfam\font\tenmsbm=msbm10\textfont
\msbmfam=\tenmsbm\font\sevenmsbm=msbm7
\scriptfont\msbmfam=\sevenmsbm\def\bb#1{{\fam\msbmfam #1}}

\def\EE{\mathbb{E}}\def\KK{\bb K}

\def\RR{\mathbb{R}}

\def\bR{\mathbb{R}}

\def\cB{\mathcal{B}}

\def\cF{\mathcal{F}}\def\cG{\mathcal{G}}

\def\cP{\mathcal{P}}
\def\cX{\mathcal{X}}

\def\al{{\alpha}}\def\be{{\beta}}\def\de{{\delta}}
\def\ep{{\epsilon}}\def\ga{{\gamma}}
\def\la{{\lambda}}\def\om{{\omega}}\def\si{{\sigma}}
\def\De{{\Delta}}
\def\Ga{{\Gamma}}
\def\La{{\Lambda}}\def\Om{{\Omega}}
\def\th{{\theta}}
\def\<{\left<}\def\>{\right>}
\def\({\left(}\def\){\right)}
\newcommand{\lan}{\langle}

\newcommand{\ran}{\rangle}
\def\goto{{\rightarrow}}

\def\blemma{\begin{lemma}}\def\elemma{\end{lemma}}
 \def\bproposition{\begin{prop}}\def\eproposition{\end{prop}}
 \def\btheorem{\begin{theorem}}\def\etheorem{\end{theorem}}
 \def\bcorollary{\begin{corollary}}\def\ecorollary{\end{corollary}}

\def\beqlb{\begin{eqnarray}}\def\eeqlb{\end{eqnarray}}
 \def\beqnn{\begin{eqnarray*}}\def\eeqnn{\end{eqnarray*}}

 \def\mcr{\mathscr}

\begin{document}

\title{\large \bf  The reversibility and an SPDE for the generalized
Fleming-Viot Processes with mutation\thanks{Research of ZL is
supported partially by NSFC (10525103 and 10721091) and CJSP, JX by
NSF DMS-0906907, HL and XZ by NSERC.}}
\author{{Zenghu Li$\,^{a}$, Huili Liu$\,^{b}$,
Jie Xiong$\,^{c,d}$ and Xiaowen Zhou$\,^{b,}\thanks{Corresponding
author. E-mail address: xiaowen.zhou@concordia.ca
Tel: (001)514-848-2424, ext. 3220   Fax: (001)514-848-2831.}$}\\
\,\\
\small{$\;^a\,$Laboratory of Mathematics and Complex Systems, School
of Mathematical Sciences,}\\\small{ Beijing Normal University,
Beijing, 100875, PRC}\\
\small{$\;^b\,$Department of Mathematics and Statistics, Concordia
University,}\\
\small{1455 de Maisonneuve  Blvd. West, Montreal, Quebec, H3G 1M8,
Canada}\\
 \small{$\;^c\,$Department of Mathematics,
University of Tennessee, Knoxville, TN 37996-1300, USA}\\
\small{$\;^d\,$Department of Mathematics, Hebei Normal University,
Shijiazhuang, 050024, PRC}}
\date{\today}
 \maketitle

\date{}
\maketitle

\begin{abstract}
The $(\Xi, A)$-Fleming-Viot process  with mutation is a probability-measure-valued
process whose moment dual is similar to that of the classical
Fleming-Viot process except that the Kingman's coalescent is
replaced by the $\Xi$-coalescent, the coalescent with simultaneous
multiple collisions. We first prove the existence of such a process
for general mutation generator $A$. We then investigate its
reversibility.
We also study both the
weak and strong uniqueness of solution to the associated stochastic partial
differential equation.
  \vspace*{.125in}


\vspace*{.125in}

\end{abstract}

\bigskip


\section{Introduction}

\setcounter{equation}{0}

The classical Fleming-Viot process is a probability-measure-valued
process for mathematical population genetics.  It describes the
evolution of relative frequencies for different types of alleles in
a large population undergoing resampling together with possible
mutation, selection and recombination; see Ethier and Kurtz
\cite{EtK93} and references therein for earlier work on the classical
Fleming-Viot process. When the classical Fleming-Viot process only involves
mutation and resampling, it is well-known that its moment dual  is a function-valued Markov process
governed by the Kingman's coalescent and the mutation semigroup.

During the past ten years, more general coalescents have
been proposed and studied by many authors. For examples, the
$\La$-coalescent (cf. Pitman \cite{P} and Sagitov \cite{Sa99}) is a
coalescent with possible multiple collisions and the
$\Xi$-coalescent (cf. M$\ddot{\text{o}}$hle and Sagitov
\cite{MoSa01} and Schweinsberg \cite{Sch00B}) is a coalescent with
possible simultaneous multiple collisions. It is then interesting to
know whether there exists a generalized Fleming-Viot type
probability-measure-valued process whose dual is a function-valued
process evolving in the same way as the classical Fleming-Viot dual
but with the Kingman's coalescent replaced by the $\Xi$-coalescent.

Such a generalized Fleming-Viot process was first considered by
Donnelly and Kurtz \cite{DoKu99} and Hiraba \cite{Hi00}.
 When the spatial motion of the particle is negated, namely, the
mutation is $0$, it has also been studied  by Bertoin and Le Gall
(\cite{BeL00}, \cite{BeL03}, \cite{BeL05}, \cite{BeL06}) and Birkner
et al \cite{BiBlCaEtMoScWa05}. In particular, a special form of such
process is constructed in \cite{BeL05} using the weak solution flow
of a stochastic equation driven by a Poisson random measure.
Generalized Fleming-Viot processes with parent independent jump
mutation operators are constructed by Dawson and Li \cite{DaL10} as
strong solutions of stochastic equations driven by time-space white
noises and Poisson random measures. The classical Fleming-Viot
process with Laplacian mutation operator is characterized by Xiong
\cite{X} as the strong solution of an SPDE driven by a time-space
white noise. The common feature of the approaches of \cite{BeL05},
\cite{DaL10} and \cite{X} is to consider the processes of
distributions of the measure-valued processes instead of their
density processes. In fact, the processes studied in \cite{BeL05}
and \cite{DaL10} are usually not absolutely continuous. Similar
stochastic equations for Dawson-Watanabe superprocesses have also
been studied in \cite{BeL05}, \cite{BeL06}, \cite{DaL10} and
\cite{X}.

 This problem is also studied in the recent work of Birkner et al
\cite{MJ}. When the mutation generator $A$  is the generator for a
pure jump Markov process, two constructions of the $(\Xi,
A)$-Fleming-Viot process are found in \cite{MJ}. One
construction is based on modification of the lookdown scheme of
\cite{DoKu99} applied to exchangeable particle
systems for the classical Fleming-Viot process. The $(\Xi,
A)$-Fleming-Viot process arises as the pathwise almost sure limit of
the empirical
 measure for
the exchangeable particle system. The other construction is based on the Hille-Yosida theorem. The
resulted process gives an example of probability-measure-valued superprocess of jump
diffusion type.

In this paper, we further study the existence and various properties
of this generalized Fleming-Viot process. We first formulate a well-posed martingale problem and show that
the $(\Xi, A)$-Fleming-Viot process $X$
with general mutation generator is the unique solution to such a martingale
problem. We then show that $X$ has a unique invariant measure if the
mutation process allows a unique invariant measure.

The reversibility of a population genetic model  is an important issue for
statistical inference. The reversibility for the classical
Fleming-Viot process has been investigated in Li et al \cite{LSY99}
using Dirichlet forms and in Handa \cite{Han02} and Schmuland and
Sun \cite{ScSu02} via cocycle identity. The reversibility for an
interacting classical Fleming-Viot process is studied in Feng et al
\cite{FSVZ11}. We also consider the reversibility for $(\Xi,
A)$-Fleming-Viot process in this paper.   By adapting the approach
of \cite{LSY99} we first show that for the $(\Xi, A)$-Fleming-Viot
process $X$ to be reversible, the mutation generator $A$ is
necessary a parent independent jump generator. Furthermore, if the type
space contains at least three points or it contains two points with
non identical mutation rates to them, to be reversible the
$\Xi$-coalescent for the Fleming-Viot process has to degenerate into
the Kingman's coalescent. When the type space contains exactly two
points with equal mutation rates to them, we show that the
above-mentioned result is still valid for several examples where
explicit computations can be carried out.

When the mutation generator $A$ is the one-dimensional Laplacian
operator, we further study the SPDEs  associated with the $(\Xi,
A)$-Fleming-Viot process. In order to establish the strong
uniqueness we associate the SPDE to a backward SDE and then prove
the pathwise uniqueness of the backward SDE using a Yamada-Watanabe
type argument. Such an approach was first proposed in \cite{X} to
prove the strong uniqueness of SPDE arising from super-Brownian
motion and Fleming-Viot process over the real line.

The rest of this article is organized as follows. In Section 2 we
first introduce the $(\Xi,A)$-coalescent, which serves as the dual
to the $(\Xi,A)$-Fleming-Viot process. In Section \ref{sec2}, we
give a new construction of the $(\Xi,A)$-Fleming-Viot with general mutation
generator $A$. In Section \ref{sec3}, we study the ergodicity and
the reversibility of this process. Finally, in Section \ref{sec5} we
study an SPDE associated to the $(\Xi,A)$-Fleming-Viot process with
$A$ being the Laplacian operator. We prove the strong uniqueness of
the solution to this nonlinear SPDE driven by a Brownian sheet and a
Poisson random measure.

\section{The $(\Xi,A)$-coalescent}

We first borrow some notation from Bertoin \cite{Ber}. Put
$[n]:=\{1,\ldots,n\}$ and $[\infty]:=\{1,2,\ldots\}$. A {\it
partition} of $D\subset [\infty]$ is a countable collection
$\pi=\{\pi_i, i=1,2,\ldots\}$ of disjoint {\it blocks} such that
$\cup_{i}\pi_i=D$ and $\min\pi_i<\min\pi_j$ for $i<j$. Let
$\mathcal{P}_n$ denote the set of partitions of $[n]$ and
$\mathcal{P}_\infty$ denote the set of partitions of $[\infty]$.
Write $\mathbf{0}_{[n]}:=\{\{1\},\ldots,\{n\}\}$ for the partition
of $[n]$ consisting of singletons.

Given a partition $\pi\in\mathcal{P}_n$ for some $n$ and
$\pi'\in\mathcal{P}_{k}$ with $|\pi|\leq k$ where $|\pi|$ denotes the cardinality
of $\pi$, the {\it coagulation} of $\pi$ by $\pi'$, denoted by
$\text{Coag}(\pi,\pi')$, is defined as the following partition of
$[n]$,
\[\pi''=\left\{\pi''_j:=\cup_{i\in \pi'_j}\pi_i: j=1,\ldots,|\pi'|\right\}.\]
For example, for $\pi=\{\{1,3\}, \{2\}, \{4,5,9\}, \{6, 8\},
\{7\}\}$ and $\pi'=\{\{1,5,6\}, \{2,3,4\}\}$, we have
\[\text{Coag}(\pi,\pi')=\left\{\{1,3,7\}, \{2,4,5,6,8,9\}\right\}.\]

Given a partition $\pi$ with $|\pi|=b$ and a sequence of positive
integers $s,k_1,\ldots,k_r$ such that $k_i\geq 2, i=1,\ldots,r$
and $b=s+\sum_{i=1}^r k_i $, we say a partition $\pi''$ is obtained
by a {\it $(b;k_1,\ldots,k_r,s)$-collision} of $\pi$ if
$\pi''=\text{Coag}(\pi,\pi')$ for some partition $\pi'$ such that
\[\{|\pi'_i|: i=1,\ldots,|\pi'|
\}=\{k_1,\ldots,k_r,k_{r+1},\ldots,k_{r+s}\},\] where
$k_{r+1}=\cdots=k_{r+s}=1$, i.e. $\pi''$ is a merger of the $b$
blocks of $\pi$ into $r+s$ blocks in which $s$ blocks remain
unchanged and the other $r$ blocks contain $k_1,\ldots,k_r$ blocks
from $\pi$.

The {\it $\Xi$-coalescent} is a $\cP_\infty$-valued coalescent
 $\Pi_\infty=(\Pi_\infty(t))_{t\geq 0}$ starting from partition
$\Pi_\infty{(0)}\in \cP_\infty$ such that
 for any $n\in[\infty]$, its
restriction to $[n]$, $\Pi_n=(\Pi_n(t))_{t\geq 0}$ is a Markov chain
and that given $\Pi_n(t)$ has $b$-blocks, each
$(b;k_1,\ldots,k_r;s)$-collision occurs at rate
$\la_{b;k_1,\ldots,k_r;s}$. For the $\Xi$-coalescent to be well
defined, it is sufficient and necessary that there is a nonzero
finite measure $\Xi=\Xi_0+\sigma^2\delta_\mathbf{0}$ on the infinite simplex
\[\Delta=\left\{\mathbf{x}=(x_1,x_2,\ldots): x_1\geq x_2\geq\cdots\geq 0, \sum_{i=1}^\infty x_i\leq 1\right\}\]
such that $\Xi_0$ has no atom at $\mathbf{0}$, $\delta_\mathbf{0}$ denotes a point
mass at $\mathbf{0}$, $\sigma\geq 0$ is a constant and
\begin{equation*}\label{coal-rate}
\begin{split}
\la_{b;k_1,\ldots,k_r;s}=\sigma^2
1_{\{r=1,k_1=2\}}+\be_{b;k_1,\ldots,k_r;s},
\end{split}
\end{equation*}
where
\begin{equation}\label{be}
\be_{b;k_1,\ldots,k_r;s}=\int_\Delta \sum_{l=0}^s\sum_{i_1\neq
\cdots\neq i_{r+l}}{s\choose{l}}x_{i_1}^{k_1}\ldots
x_{i_r}^{k_r}x_{i_{r+1}} \ldots x_{i_{r+l}}\left(1-\sum_{j=1}^\infty
x_j\right)^{s-l} \frac{\Xi_0(d\mathbf{x})}{\sum_{j=1}^\infty x_j^2}
\end{equation}
denotes the rate of simultaneous multiple coalescent with $\Xi_0$
being the measure of multiple coagulation and $\sigma^2$ denotes the
rate of binary coagulation.  As a result, the coagulation rates
satisfy the {\it consistency condition}
\[\la_{b;k_1,\ldots,k_r;s}=\sum_{m=1}^r \la_{b+1;k_1,\ldots,k_{m-1},k_m+1,k_{m+1},\ldots,k_r;s}
+s\la_{b+1;k_1,\ldots,k_r,2;s-1}+\la_{b+1;k_1,\ldots,k_r;s+1}.\]
See Schweinsberg \cite{Sch00B}.

When the measure $\Xi$ is supported on $[0,1]$, the corresponding
coalescent  involves at most multiple collisions. Such
$\Xi$-coalescents are also called $\Lambda$-coalescent.

A Poisson process construction of $\Xi$-coalescent is  given in \cite{Sch00B} as follows. For each $\mathbf{x}=(x_1,x_2,\ldots)\in \Delta $ let $P_\mathbf{x}$ be a probability measure on $\mathbb{Z}^\infty$ such that \[P_\mathbf{x}\{\mathbf{z}=(z_1,\ldots)\in\mathbb{Z}^\infty:z_i=j \}=x_j,\]
\[P_\mathbf{x}\{\mathbf{z}: z_i=-i\}=1-\sum_{j=1}^\infty x_j\] and for any $n\in [\infty]$ and $(k_1,\ldots,k_n)\in \mathbb{Z}^n$,
\[P_\mathbf{x}\{\mathbf{z}: z_i=k_i, i=1,\ldots,n\}=\prod_{i=1}^n P_\mathbf{x}\{\mathbf{z}: z_i=k_i \}.\]
Then define a measure $L$ on $\mathbb{Z}^\infty$ by
\[L(B):=\int_\Delta P_\mathbf{x}(B)\frac{\Xi_0(d\mathbf{x})}{\sum_{j=1}^\infty x_j^2}+\sigma^2 \sum_{i=1}^\infty\sum_{j=i+1}^\infty 1_{\mathbf{z}_{ij}\in B}\]
where $\mathbf{z}_{ij}=(z_1,\ldots)$ with $z_i=z_j=1$ and $z_k=-k$ for $k\not\in \{i,j\}$. Let $(e(t))_{t\geq 0}$ be a $\mathbb{Z}^\infty$-valued Poisson point process with characteristic measure $L$. Let $(e_n(t))_{t\geq 0}$ be the process $(e(t))_{t\geq 0}$ restricted to $\mathbb{Z}^n$.  Notice that $e$ and $e_n$ can be identified as $\mathcal{P}_\infty$-valued and $\mathcal{P}_n$-valued process, respectively, in the obvious way. Then the $\Xi$-coalescent $\Pi_n$ can be constructed using $e_n$ recursively. Given $\Pi_n(0)\in\mathcal{P}_n$ and suppose that $\Pi_n(s)$ has been constructed for $0\leq s\leq t$. Let $T>t$ be the first jumping time for $e_n$ after $t$. Then define $\Pi_n(s)=\Pi_n(t)$ for $t<s<T$ and $\Pi_n(T)=\text{Coag} (\Pi_n(t),e_n(T))$.
 See  \cite{Sch00B} and references therein  for more detailed inductions on the  $\Xi$-coalescent.

Let $E$ be a Polish space containing at least two points. Let $B(E)$
be the set of bounded functions on $E$. Let $M_1(E)$ be the space of
Borel probability measures on $E$. Then $M_1(E)$ endowed with the
topology of weak convergence is a Polish space. For $\mu\in M_1(E)$
and $f\in B(E)$, write $\<\mu,f\> =\mu(f):= \int f d\mu$. For $n\ge
1$ and $f\in B(E^n)$, let
 \beqlb\label{1.1}
G_{n,f}(\mu)
 =
G_\mu(n,f)
 =
\<\mu^n,f\>.
 \eeqlb
Let $Z$ be a Markov process with state space $E$, Feller transition
semigroup $(P_t)$ and generator $(A, \mathcal{D}(A))$. Process $Z$ describes the mutation mechanism for the Fleming-Viot process.

Given a function $g$ on $E^n$, for each partition $\pi=\{\pi_i,
i=1,\ldots,|\pi|\}$ on $[n]$, we define $\Phi_\pi g$ as a function
on $E^{|\pi|}$ such that $\Phi_\pi
g(x_1,\ldots,x_{|\pi|})=g(x_{i_1},\ldots,x_{i_n})$ with $i_j=k$ for
$i_j\in \pi_k, j\in [n]$. For example, if $g$ is a function on
$E^6$ and $\pi=\{\{1,4\},\{2,3,6\},\{5\}\}$, then \[\Phi_\pi
g(x_1,x_2,x_3)=g(x_1,x_2,x_2,x_1,x_3,x_2).\]  For $1\leq i<j\leq n$
we write $\Phi_{ij}$ for $\Phi_\pi$ with
\[\pi=\left\{\{1\},\ldots,\{i-1\},\{i,j\},\{i+1\},\ldots,\{j-1\},\{j+1\},\ldots,\{n\}\right\}.\]

The {\it $(\Xi, A)$-coalescent} $(M,Z)=(M_t,Z_t)_{t\geq 0}$ with
initial value $(M_0,Z_0)=(n,f), f\in B(E^n)$ is a $ \cup_{k=1}^n
\{k\}\times B(E^k)$-valued Markov process defined as follows. Let
$\Pi_n$ with $\Pi_n(0)=\mathbf{0}_{[n]}$ be the $\Xi$-coalescent
defined before. For any $f\in C(E^n)$, the set of bounded continuous
functions on $E^n$ equipped with the supremum norm, define semigroup
$(P_t^{(n)})$ by
\[P^{(n)}_tf(x_1,\ldots,x_n):=\int_{E^n}f(\xi_1,\ldots,\xi_n)\prod_{i=1}^n
P_t(x_i,d\xi_i).\]
 Then $M:=|\Pi_n|$ and $Z$ can be
defined iteratively as follows. Write $T_0=0$ and $T_1<T_2<\cdots$
for the sequence of ordered jumping times for $\Pi_n$. For any $T_i<t<T_{i+1}$,
define
\[Z_t:=P^{(M_{T_i})}_{t-T_i}Z_{T_i}\,\,\, \text{and}\,\,\,
Z_{T_{i+1}}=\Phi_{\pi'}Z_{T_{i+1}-}\,\,\, \text{if}\,\,\,
\Pi_n(T_{i+1})=\text{Coag} (\Pi_n(T_{i+1}-),\pi').\]

For any $\pi\in\cP_n, \pi\neq \mathbf{0}_{[n]}$, write
$\beta_{\pi}:=\beta_{n;k_1,\ldots,k_r;s}$ where
$(n;k_1,\ldots,k_r;s)$ is uniquely determined by
\[\{k_1,\ldots,k_r,1,\ldots,1\}=\{|\pi_i|: 1\leq i\leq |\pi|\}.\]

For fixed $\mu\in M_1(E)$ and the function $G_\mu$ given by
(\ref{1.1}) we define for any $f\in\mathcal{D}\(A^{(n)}\)$,
 \begin{equation*}
 \begin{split}
L^*G_\mu(n,f)
 =&
G_\mu(n,A^{(n)}f) + \sigma^2\sum_{i<j}\Big[G_\mu(n-1,\Phi_{ij}f) -
G_\mu(n,f)\Big] \\
 &
+\sum_{\pi\in\cP_n\backslash\{\mathbf{0}_{[n]}\}} \beta_{\pi}
\Big[G_\mu\(|\pi|,\Phi_{\pi}f\) - G_\mu(n,f)\Big],
 \end{split}
 \end{equation*}
 where $A^{(n)}$ denotes the generator of $\(P^{(n)}_t\)$.
Clearly, $L^*$  is the generator for the Markov process
$(M_t,Z_t)$.

\section{Existence of the $(\Xi, A)$-Fleming-Viot process}\label{sec2}
We call an $M_1(E)$-valued Markov process $X$ a $(\Xi,
A)$-Fleming-Viot process with
$\Xi$-resampling mechanism and mutation generator $A$ if for any
$n\in\mathbb{N}, f\in B(E^n)$, its moment is determined by
\begin{equation}\label{moment-dual}
\EE G_{n,f}(X_t)=\EE G_{X_0}(M_t,Z_t)
\end{equation}
 where $(M,Z)$
denotes the $(\Xi,A)$-coalescent with initial value $(n,f)$, which
is defined in the previous section. Since process $X$ is
probability-measure-valued, its distribution is uniquely determined
by (\ref{moment-dual}).

The $(\Xi, A)$-Fleming-Viot process is constructed in Birkner et al
\cite{MJ} for generator $A$  of jump type i.e.
$$Af(x)=r\int_E(f(y)-f(x))q(x,dy),$$ where $f$ is a bounded function on $E$, $q(x,dy)$ is a Feller transition
function and $r>0$ is the global mutation rate.

 In this section, we want to show that the desired probability-measure-valued process is well defined
 for any Feller generator $A$ with transition semigroup $(P_t)$ on $C(E)$. 
To this end, we want to show that the $(\Xi, A)$-Fleming-Viot
process is the unique solution to a martingale problem. Let
$D_{M_1}[0,\infty)$ be the space of c$\grave{\text
a}$dl$\grave{\text a}$g paths from $[0,\infty)$ to $M_1(E)$
furnished with the Skorohod topology.

Let $\mcr{D}_1\subset C(M_1(E))$ be the linear span of the functions
of form (\ref{1.1}) with $n\ge 1$ and $f\in \mathcal{D}(A^{(n)})$.
Let $L_0$ be the linear operator from $\mcr{D}_1$
to $C(M_1(E))$ defined by
 \beqlb\label{1.2}
L_0G_{n,f}(\mu) = \<\mu^n,A^{(n)}f\> + \sigma^2\sum_{1\leq i<j\leq n}
\Big[\<\mu^{n-1},\Phi_{ij}f\> - \<\mu^n,f\>\Big]
 \eeqlb
for $\mu\in M_1(E)$.

Let $C_{M_1}[0,\infty)$ be the space of continuous paths from
$[0,\infty)$ to $M_1(E)$ furnished with the topology of
uniform convergence. The next theorem follows from Theorem~3.2 of
Ethier and Kurtz \cite{EtK93}.

\btheorem\label{t2.1} The $(L_0,\mcr{D}_1)$-martingale problem in
$C_{M_1}[0,\infty)$ is well-posed. \etheorem

Note that the solution to the $(L_0,\mcr{D}_1)$-martingale problem
is the well-known classical Fleming-Viot process. In the generator, only the
Kingman's coalescent is represented by the term $\Phi_{ij}$. To
extend the process, we introduce more general coalescent.

Fix $n$ and $f\in B(E^n)$. For $\mu\in M_1(E)$ define
 \beqlb   \label{eq0506b}
\mathbf{B}G_{n,f}(\mu)=\sum_{\pi\in\cP_n\backslash\{\mathbf{0}_{[n]}\}
}\beta_\pi\Big[\<\mu^{|\pi|},\Phi_\pi f\> - \<\mu^n,f\>\Big].
 \eeqlb

For $f\in \mathcal{D}(A^{(n)})$ let $L$ be the linear operator from
$\mcr{D}_1$ to $C(M_1(E))$ defined by
 \beqlb \label{eq0506c}
LG_{n,f}(\mu) = L_0G_{n,f}(\mu) + \mathbf{B}G_{n,f}(\mu).
 \eeqlb

Define test functions $F(\mu)=\prod_{i=1}^{n}\left<\mu, f_i\right>$
for any $n$ and any $f_i\in B(E), i=1,\ldots,n$. The generator
related to the simultaneous multiple part of the $\Xi$- coalescent
is defined as

\[\mathbf{B} F(\mu) =  \sum_{\pi\in\cP_n\backslash\{0_{[n]}\}}\be_{\pi}
\left(\prod_{i=1}^{|\pi|}\<\mu,\prod_{j\in\pi_i}
f_j\>-\prod_{i=1}^n\<\mu, f_i\>\right). \]

The following Lemma shows a different representation for this
generator, which has  been obtained in Birkner et al \cite{MJ}
(Equation (4.2)) for  jump type  generator $A$. Its proof is
essentially the same as \cite{MJ} and we omit it.

\begin{lemma}\label{gen-jump}
The generator $\mathbf{B}$ can be represented as
$$\mathbf{B}F(\mu)=\int_{\Delta}\(\sum_{i=1}^{\infty}z_i^2\)^{-1}\Xi_0\(dz\)
\int_{E^\mathbb{N}}\left(F\(\sum_{j=1}^{\infty}z_j\delta_{x_j}+\(1-\sum_{j=1}^{\infty}
z_j\)\mu\)-F\(\mu\)\right){\mu}^{\otimes\mathbb{N}}\(dx\),$$ where
$x_i, i=1,2,3,\ldots$ are independently and identically distributed
random variables with common distribution $\mu$. $\Delta$ is the
infinite simplex satisfying $\Delta=\{z=\(z_1,z_2,\ldots\):z_1\geq
z_2\geq\cdots\geq 0,\sum_{i=1}^{\infty}z_i\leq 1\}$.\end{lemma}

Lemma \ref{gen-jump} shows that the $(\Xi,A)$-Fleming-Viot process is a jump
diffusion type superprocess such that between jumping times it
evolves like a classical Fleming-Viot process, and at a jumping time a
fraction of its mass is redistributed over the current support of
the process.

We now consider the martingale problem (MP) related to the SPDE to
be considered later in Section 5. Let $\cX$ be the closure of $D(A)\subset
C_b(E)$ with respect to the norm
\[\|f\|\equiv\|f\|_\infty+\|Af\|_\infty,\]
where $\|f\|_\infty$ is the supremum norm of $f$. Denote the dual
space of the Banach space $\cX$ by $\cX^*$. Introduce the distance
in $\cX^*$ by
\[\rho(X,Y)=\sum_{j=1}^\infty 2^{-j}\(|X(f_j)-Y(f_j)|\wedge 1\),\]
where $\{f_j\}\subset D(A)$ is such that $span(\{f_j\})$ is dense in
$\cX$.

A stochastic process $X\in D(M_1(E))$ is a solution to the following
generalized Fleming-Viot MP (GFVMP in short) if for any $f\in
\mathcal{D}(A)$,
\begin{equation}\label{eq0913a}
M^f_t=\<X_t,f\>-\<\mu,f\>-\int^t_0\<X_{s},Af\>ds
\end{equation}
is a square-integrable martingale such that
\begin{equation}\label{eq0913a2}\<M^{f,c}\>_t=\si^2\int^t_0\(\<X_{s},f^2\>-\<X_{s},f\>^2\)ds,\end{equation}
and $\forall\ B\in\Ga$, the process
\begin{equation}\label{eq0913a3}
\sum_{0<s\le t}1_B(\De
M_s)-\int^t_0\int_\Delta\int_{E^\mathbb{N}}1_B\(\sum^\infty_{i=1}z_i(\de_{x_i}-X_s)\)\ga(ds
dz dx)\end{equation} is a martingale, where $M^{f,c}$ is the
continuous part of the martingale $M^f$ and $M_t$ is the
$\cX^*$-valued martingale such that $M^f_t=M_t(f)$, $\De
M_s=M_s-M_{s-}$,
\[\ga(ds dz dx)=ds \otimes \(\sum_{i=1}^\infty
z_i^2\)^{-1}\Xi_0(dz)\otimes X_{s}^\mathbb{N}(dx)\] is a random
measure on $\RR_+\times\Delta\times E^{\mathbb{N}}$, and
\[\Ga=\left\{B\in\cB(\cX^*\setminus\{0\}):\;\forall\
t>0,\;\EE\int^t_0\int_\Delta\int_{E^\mathbb{N}}1_B\(\sum^\infty_{i=1}z_i(\de_{x_i}-X_s)\)\ga(ds
dz dx)<\infty\right\}.\]

\begin{theorem}\label{SPDE-weak}
The GFVMP has a solution in $D_{M_1}[0,\infty)$. Further, every
solution $X$ to the GFVMP is a $(\Xi,A)$-Fleming-Viot process.
Consequently, the GFVMP is well-posed.
\end{theorem}

\begin{proof}
If $X$ is a solution to the GFVMP, then
\[\sum_{0<s\le t}\lan\De M_s,f\ran-\int^t_0\int_\Delta\int_{E^\mathbb{N}}\sum^\infty_{i=1}z_i(f(x_i)-\lan X_s,f\ran)\ga(ds
dz dx)\] is a pure jump martingale. Then (\ref{moment-dual}) follows
from  It\^o's formula (Theorem 4.57 of \cite{JS}),  Lemma
\ref{gen-jump} and the martingale duality argument. The solution $X$
is thus a $(\Xi,A)$-Fleming-Viot process.

Now we proceed to prove the existence of a solution for this
martingale problem. First, we suppose that $(\sum_{i=1}^\infty
z_i^2)^{-1}\Xi_0(dz)$ is a finite measure. Without loss of
generality, we may assume that the total mass is 1. Let
$0<\tau_1<\tau_2<\cdots$ be the jump times of a standard Poisson
process. Let $X_t$ be an $M_1(E)$-valued process defined as follows.
Between the Poisson times, it evolves like the usual Fleming-Viot process. At
the jump time $\tau_j$, we independently choose a $\Delta$-valued
random variable $(z_i)$ with distribution $(\sum_{i=1}^\infty
z_i^2)^{-1}\Xi_0$, and an $E^\mathbb{N}$-valued random variable
$(x_i)$ with distribution $X_{\tau_j-}^{\otimes\mathbb{N}}$,
$j=1,2,\ldots$ and then set
\[X_{\tau_j}:=\sum_{i=1}^\infty z_i\delta_{x_i}+\left(1-\sum_{i=1}^\infty z_i\right)X_{\tau_j-}.\]
It is then easy to verify that $X_t$ is a solution
to the martingale problem (\ref{eq0913a}).

In general, we approximate $\Xi_0$ by
$\nu_n(dz)=1_{\sum_{i=1}^\infty z_i^2>1/n}\Xi_0(dz)$. Let $X^n$ be
the solution to the GFVMP with $\Xi_0$ replaced by $\nu_n$. Then,
\[M^n_t=X^n_t-\mu-\int^t_0 A^*X^n_sds\]
is a sequence of $\cX^*$-valued martingales. Denote its continuous
and purely-discontinuous parts by $M^{n,c}$ and $M^{n,d}$,
respectively.    Then
\[\<M^{n,c}(f)\>_t=\si^2\int^t_0\(\<X^n_s,f^2\>-\<X^n_s,f\>^2\)ds\]
and for any $B\in\Ga$, the process
\[\sum_{0<s\le t}1_B(\De
M^n_s)-\int^t_0\int_\Delta\int_{E^\mathbb{N}}1_B\(\sum^\infty_{i=1}z_i(\de_{x_i}-X^n_s)\)\ga^n(ds
dz dx)\] is a martingale, where $A^*$ is the adjoint operator of
$A$. Note that $X^n$ is a sequence of $\cX^*$-valued processes. To
prove the tightness of $X^n$, we only need to prove the tightness of
$\<X^n,f\>$ for any $f\in \cX$. Note that the finite variation part
is
\[A^n_t=\int^t_0\<X^n_s,Af\>ds\]
and the martingale part $M^n_t(f)$ has quadratic variation processes
\begin{eqnarray*}
\<M^n(f)\>_t&=&\(\si^2+\Xi_0(\Delta)\)\int^t_0\(\<X^n_s,f^2\>-\<X^n_s,f\>^2\)ds\\
&&+\int^t_0\int_\Delta\int_{E^\mathbb{N}}\(\sum^\infty_{i=1}z_i(f(x_i)-X^n_s(f))\)^2\ga^n(ds
dz dx) .\end{eqnarray*} It is easy to prove that $\{A^n\}$ and
$\{\<M^n(f)\>\}$ are $C$-tight. By Corollary 3.33 (p317) and Theorem
4.13 (p322) of Jacod and Shiryaev \cite{JS} (see also Theorem 6.1.1
of Kallianpur and Xiong \cite{KX}) we see that $\<X^n,f\>$ is tight.
So, $X^n$ is tight. Similarly, we can prove the tightness of
$\{(M^{n,c},M^{n,d})\}$.

Denote a limit of $\{(X^n,M^{n,c},M^{n,d})\}$ by
$(X,\tilde{M}^1,\tilde{M}^2)$. It is easy to prove that
$\tilde{M}^1(f)$ is a continuous martingale,
$M_t=\tilde{M}^1_t+\tilde{M}^2_t$ and
\[\<\tilde{M}^1(f)\>_t=\si^2\int^t_0\(\<X_s,f^2\>-\<X_s,f\>^2\)ds.\]
By the same arguments as those in the proofs of Theorem 6.1.3 and
Lemma 6.1.11 of \cite{KX}, we can prove that $\{\tilde{M}^2(f)\}$ is
purely-discontinuous, and for any $B\in\Ga$, the process
\[\sum_{0<s\le t}1_B(\De\tilde{M}^2_s)-\int^t_0\int_{\Delta}\int_{E^{\mathbb{N}}}
1_B\(\sum_{i=1}^\infty z_i\de_{x_i}-\sum_{i=1}^\infty z_iX_s\)\ga(ds
dz dx)\] is a martingale. Thus, $(X_t)$ is a solution to the
martingale problem (\ref{eq0913a}).
 \qed
\end{proof}

\section{The reversibility}\label{sec3}

In this section, we consider the reversibility of the $(\Xi,A)$-Fleming-Viot
process whose existence is justified in the previous section. We
will prove that it  is irreversible except for the case of
classical Fleming-Viot process with parent independent mutation.

To obtain the existence and uniqueness of the invariant measure for
the $(\Xi,A)$-Fleming-Viot process we need the following assumption.

\noindent {\bf Assumption (I)}: The Markov process $Z$ with
generator $A$ has a unique invariant measure
 $\nu\in M_1(E)$ such that $P_t^*\mu\to\nu$  weakly for any $\mu\in M_1(E)$ as $t\rightarrow\infty$, where $(P^*_t)$
denotes the adjoint for $(P_t)$.

\begin{lemma}
Suppose that the Assumption (I) holds. Then,

 (a) The $(\Xi, A)$-Fleming-Viot process
$X$ has at least one invariant measure.

(b) Let $\Pi$ be any invariant measure of the $(\Xi, A)$-Fleming-Viot process
$X$. For any positive integer $n$ and any $f\in B(E^n)$, we have
\[\int_{M_1(E)}\<\mu^n,f\>\Pi(d\mu)=\EE_{(n,f)}\<\nu, Z_\tau\>,\]
where \[\tau:=\inf\{t\geq 0: |M_t|=1\}\] and $(M,Z)$ is the
$(\Xi, A)$-coalescent  starting at $(n,f)$. Consequently, $X$
has a unique invariant measure.
\end{lemma}

\begin{proof} (a) As $P_t^*X_0\to\nu$, the family $\{P_t^*X_0:\ t\ge
0\}$ is pre-compact, and hence, tight in $M_1(E)$. Thus, for any
$\ep>0$, there exists a compact subset $K_\ep$ of $E$ such that
$P_t^*X_0(K^c_\ep)<\ep$ for all $t\ge 0$. Let
\[\KK_\ep=\left\{\rho\in M_1(E):\;\rho(K^c_{\ep k^{-1}2^{-k}})\le
k^{-1},\;\;\forall\ k\ge 1\right\}.\] For any $\de>0$, let $k\ge 1$
be such that $k^{-1}<\de$. Then, for any $\rho\in\KK_\ep$, we have
\[\rho(K^c_{\ep k^{-1}2^{-k}})\le
k^{-1}<\de,\] and hence, $\KK_\ep$ is tight in $M_1(E)$. Then,
$\KK_\ep$ is a pre-compact subset of $M_1(E)$.

Note that
\begin{eqnarray*}
T^{-1}\int^T_0PX^{-1}_tdt(\KK^c_\ep)&=&T^{-1}\int^T_0P\(\exists\
k\ge
1,\;\;X_t(K^c_{\ep k^{-1}2^{-k}})>k^{-1}\right)dt\\
&\le&T^{-1}\int^T_0\sum^\infty_{k=1}k\EE X_t(K^c_{\ep k^{-1}2^{-k}})dt\\
&=&T^{-1}\int^T_0\sum^\infty_{k=1}k P_t^*X_0(K^c_{\ep k^{-1}2^{-k}})dt\\
&<&T^{-1}\int^T_0\sum^\infty_{k=1}k \ep k^{-1}2^{-k}dt=\ep.
\end{eqnarray*}
Thus, the family  $\left\{T^{-1}\int^T_0PX^{-1}_tdt:\;T\ge
0\right\}$ is tight, and hence, pre-compact in $M_1(M_1(E))$. Let
$\Pi$ be a limit point. Then there exists a sequence $(T_n)$ such
that $T_n\uparrow \infty$ and
\[\lim_{n\goto\infty }T_n^{-1}\int^{T_n}_0PX^{-1}_t dt=\Pi.\]
For any $s\ge 0$,
\begin{eqnarray*}
\Pi X^{-1}_s&=&\lim_{n\to\infty}T_n^{-1}\int^{T_n}_0PX^{-1}_t\circ
X^{-1}_s
dt\\
&=&\lim_{n\to\infty}T_n^{-1}\int^{T_n+s}_sPX^{-1}_{t}dt\\
&=&\Pi.
\end{eqnarray*}
Namely, $\Pi$ is an invariant measure of the stochastic process
$\{X_t\}$.

(b) Note that $\tau<\infty$. By the moment duality
(\ref{moment-dual}) and the strong Markov property we have
\begin{equation*}
\begin{split}
\int_{M_1(E)} \<\mu^n,f\>\Pi(d\mu)&=\lim_{t\goto\infty}\EE_\Pi\< X_t^n, f\>\\
&=\lim_{t\goto\infty}\int_{M_1(E)}\EE_{(n,f)}[G_\mu(M_t,Z_t)]\Pi(d\mu)\\
&=\lim_{t\goto\infty}\int_{M_1(E)}\EE_{(n,f)}[G_\mu(M_t,Z_t)1_{\tau\leq t}]\Pi(d\mu)\\
&=\lim_{t\goto\infty}\int_{M_1(E)}\EE_{(n,f)}\EE_{(1,Z_{\tau})}[\<\mu,Z_{t-\tau}\>1_{\tau\leq t}]\Pi(d\mu)\\
&=\lim_{t\goto\infty}\int_{M_1(E)}\EE_{(n,f)}\EE_{(1,Z_{\tau})}\<\mu,
Z_t\>\Pi(d\mu)\\
&=\lim_{t\goto\infty}\int_{M_1(E)}\EE_{(n,f)}\<\mu,
P_tZ_{\tau}\>\Pi(d\mu)\\
&=\lim_{t\goto\infty}\int_{M_1(E)}\EE_{(n,f)}\<P^*_t\mu,
Z_{\tau}\>\Pi(d\mu)\\
&=\EE_{(n,f)}\<\nu, Z_{\tau}\>,
\end{split}
\end{equation*}
where we have used the fact that $\tau<\infty$ a.s.. \qed
\end{proof}

We now consider the reversibility of the $(\Xi, A)$-Fleming-Viot process $X$.

The transition semigroup $(P_t)$ is irreducible if for any $x\in
E$ and $f\in C_0^+(E)$ with $f\neq 0$, $P_tf(x)>0$ for some
$t>0$, where $C_0^+(E)$ denotes the space of nonnegative and continuous functions vanishing at infinity.

\begin{lemma}\label{rev1}
Suppose that $(P_t)$ is irreducible.  If the $(\Xi, A)$-Fleming-Viot process
$X$ is reversible, then its mutation generator $A$ is a parent
independent pure jump generator, i.e.
\[Af(\cdot)=\frac{\theta}{2}\int_E(f(y)-f(\cdot))\nu_0(dy)\] for some
$\theta>0$ and probability measure $\nu_0$ on $E$.
\end{lemma}

\begin{proof}
 The result of the current Lemma has been proved for the classical Fleming-Viot process by Li et al \cite{LSY99}, Handa \cite{Han02} and Schmuland and Sun \cite{ScSu02} with different methods. It is simple to see that the barycenter $m$ of a stationary distribution of the classical Fleming-Viot process is a stationary distribution for its mutation operator $A$. A key observation in \cite{LSY99} is that the reversibility of $A$ relative to $m$ is structured by the first three moment measures of the classical Fleming-Viot process. The necessary form of $A$ was then derived in \cite{LSY99} from Beurling-Deny formula of the associated Dirichlet form.

We shall see that the structures explained above are maintained by the generalized Fleming-Viot processes. Our proof is an adaption of the approach for Theorem 1.1 of \cite{LSY99}. The results of Lemmas 2.1, 2.4, 2.5 and 2.6 in \cite{LSY99} still hold for the generalized Fleming-Viot process. We only need to make some modifications in the first half of the proof of their Lemma 2.1 and the second half of the proof of their Lemma 2.5 as follows. By possible time rescaling, without loss of generality we assume $\sigma^2+\beta_{2;2;0}=1$ in the sequel of this proof.

For the proof of the analog of Lemma 2.1 of \cite{LSY99}, let
$\varPi$ be the reversible measure of $X$ and define the moment
measures
\[m:=\EE_\varPi X_t, \, m_2:=\EE_\varPi
X_t\otimes X_t \text{\,\,\, and \,\,\,} m_3:=\EE_\varPi
X_t\otimes X_t\otimes X_t,\]
where we adopt the notation of \cite{LSY99}.
Given any partition $0=t_0<t_1<\cdots<t_n=t$ and $f,g,h\in \mathcal{D}(A)$, by
the invariance of $\varPi$ and the moment duality we have
\begin{equation*}
\begin{split}
&m_2(P_{t_i}f\otimes P_{t_i+r}g)\\
&=\EE_\varPi\left[ X_{t_{i+1}-t_i}(P_{t_i}f)X_{t_{i+1}-t_i}(P_{t_i+r}g) \right]\\
&=e^{-(\sigma^2+\beta_{2;2;0})(t_{i+1}-t_i)}\EE_\varPi\left[ X_0(P_{t_{i+1}}f)X_0(P_{t_{i+1}+r}g)\right]\\
&\quad+\int_0^{t_{i+1}-t_i} ds(\sigma^2+\beta_{2;2;0})e^{-(\sigma^2+\beta_{2;2;0})s}\EE_\varPi X_0(P_{t_{i+1}-t_i-s}(P_{t_i+s}fP_{t_i+s+r}g))\\
&=e^{-(t_{i+1}-t_i)}m_2(P_{t_{i+1}}f\otimes
P_{t_{i+1}+r}g)+\int_0^{t_{i+1}-t_i} ds
e^{-s}m(P_{t_{i+1}-t_i-s}(P_{t_i+s}fP_{t_i+s+r}g)).
\end{split}
\end{equation*}
Letting $n\goto\infty$ and $t_{i+1}-t_i=t/n$, we have
\begin{equation*}
\begin{split}
& m_2(f\otimes P_rg)-m_2(P_tf\otimes P_{t+r}g)\\
&=\lim_{n\goto\infty} \sum_{i=0}^{n-1}\left(m_2(P_{t_i}f\otimes P_{t_i+r}g)-m_2(P_{t_{i+1}}f\otimes P_{t_{i+1}+r}g)\right)\\
&=\int_0^t m(P_sfP_{s+r}g)ds-\int_0^t m_2(P_sf\otimes P_{s+r}g)ds.
\end{split}
\end{equation*}
We have thus obtained (2.4) of \cite{LSY99}.

For the proof of the analog of Lemma 2.5 of \cite{LSY99} by the
reversibility and the moment duality again,
 \begin{equation}\label{rev}
 \begin{split}
 &\EE_\varPi[X_0^{\otimes 2}(f\otimes g)X_0(P_th-h)]\\
 &=\int \Pi(d\mu)\EE_\mu [(X_t^{\otimes 2}(f\otimes g)-X_0^{\otimes 2}(f\otimes g))X_0(h)] \\
 &=\int \Pi(d\mu)\EE_\mu\left[X_0(h)\left(e^{-(\sigma^2+\beta_{2;2;0})t}X_0^{\otimes 2}(P_tf\otimes P_tg)\right.\right.\\
 &\qquad\qquad\quad\left.\left.+ \int_0^t ds (\sigma^2+\beta_{2;2;0})e^{-(\sigma^2+\beta_{2;2;0})s}X_0(P_{t-s}(P_sfP_sg))ds-X_0^{\otimes 2}(f\otimes g)  \right)\right].
 \end{split}
 \end{equation}
Dividing both sides of (\ref{rev}) by $t$ and letting $t\rightarrow
0+$, we have
\begin{equation*}
\begin{split}
&\EE_\varPi[X_0^{\otimes 3}(f\otimes g\otimes Ah)]\\
&=\EE_\varPi[X_0^{\otimes 3}(Af\otimes g\otimes h+f\otimes Ag\otimes h-
(\sigma^2+\beta_{2;2;0})f\otimes g\otimes h)+X_0^{\otimes
2}((\sigma^2+\beta_{2;2;0})(fg)\otimes h)]\\
&=\EE_\varPi[X_0^{\otimes 3}(Af\otimes g\otimes h+f\otimes Ag\otimes h-
f\otimes g\otimes h)+X_0^{\otimes 2}((fg)\otimes h)].
\end{split}
\end{equation*}
So, \[m_3(f\otimes g\otimes Ah)=m_3(Af\otimes g\otimes h+f\otimes Ag\otimes h-
f\otimes g\otimes h)-m_2((fg)\otimes h).\]
The rest of the proof then follows from that in Lemma 2.5 of
\cite{LSY99}.

The same argument for Theorem 1.1 of  \cite{LSY99} can also go
through for the generalized Fleming-Viot process.\qed
\end{proof}

\begin{lemma}\label{lem0928a}
If $\be_{p;p;0}=0$  for some $p\geq 3$, then the $\Xi$-coalescent
 degenerates to Kingman's coalescent.
\end{lemma}

\begin{proof}
For any $p\geq 3$, by (\ref{be}) we have
\[\be_{p;p;0}
 =\int_\Delta\(\sum_{i=1}^{\infty}x_{i}^p\)\(\sum_{i=1}^{\infty}x_i^2\)^{-1}\Xi_0\(d\mathbf{x}\).\]
If $\be_{p;p;0}=0$, then $\Xi_0$ must be a zero measure
since $\Xi_0$ has no atom at zero, and the integrand is always
positive except at zero. Therefore, the $\Xi$-coalescent  degenerates to Kingman's
coalescent.
\qed
\end{proof}

\begin{theorem}\label{1117}
Suppose that $(P_t)$ is irreducible and space $E$ contains at least
three different points or $E=\left\{e_1,e_2\right\}$ with
$\nu_0\(\{e_1\}\)\neq\nu_0\(\{e_2\}\)$. Then the $(\Xi,A)$-Fleming-Viot
process $X$ is reversible if and only if it is the classical Fleming-Viot
process with parent independent mutation.
\end{theorem}

\begin{proof}
 If the generalized Fleming-Viot process $X$ is reversible, then there exists an invariant measure,
say $\Pi$, satisfying
\begin{equation}\label{eq0915c}
\int_{M_1(E)}G_{p,f}(\mu)LG_{q,g}(\mu)\Pi(d\mu)=\int_{M_1(E)}G_{q,g}(\mu)LG_{p,f}(\mu)\Pi(d\mu),\end{equation}
where $p$ and $q$ are arbitrary nonnegative integers.

By Lemma \ref{rev1} the generator $A$ is a parent independent jump
operator. Also note that $\nu_0$ has a support $E$ since $(P_t)$ is
irreducible. If $E$ contains at least three different points or
$E=\left\{e_1,e_2\right\}$ with
$\nu_0\(\{e_1\}\)\neq\nu_0\(\{e_2\}\)$, we can choose $F\subset E$
such that
$$\nu_0(F)=\al\neq 0, 1/2, 1.$$
 Taking $p=1$, $q=0$ and $f=1_F$ in (\ref{eq0915c}) we get
\begin{equation}\label{eq0920h}
\int_{M_1(E)}\mu(F)\Pi(d\mu)=\al.
\end{equation}
Taking $p=2$, $q=0$ and $f=1_{F\times F}$, we get
\begin{equation}\label{eq0920a}
\int_{M_1(E)}\mu^2(F)\Pi(d\mu)=\frac{\si^2+\be_{2;2;0}+\th\al}{\si^2+\be_{2;2;0}+\th}\al.
\end{equation}
Taking $p=3$, $q=0$ and $f=1_{F\times F\times F}$, we get
\begin{equation}\label{eq0915a}
\int_{M_1(E)}\mu^3(F)\Pi(d\mu)=\frac{\be_{3;3;0}\al+\frac32\(2\si^2+2\be_{3;2;1}+\th\al\)
\frac{\si^2+\be_{2;2;0}+\th\al}{\si^2+\be_{2;2;0}+\th}\al}{3\si^2+3\be_{3;2;1}+\be_{3;3;0}+\frac32\th}.\end{equation}
Taking $p=1$, $q=2$, $f=1_F$ and $g=1_{F\times F}$, we get
\begin{equation}\label{eq0915b}
\int_{M_1(E)}\mu^2(F)\Pi(d\mu)\times\(\be_{2;2;0}+\si^2+\frac{\th\al}{2}\)-\int_{M_1(E)}\mu^3(F)\Pi(d\mu)\times\(\be_{2;2;0}+\si^2+\frac{\th}{2}\)=0.\end{equation}
Plugging in Equations (\ref{eq0920a}) and (\ref{eq0915a}). Since $\be_{3;2;1}=\be_{2;2;0}-\be_{3;3;0}$ (consistency
condition), then Equation  (\ref{eq0915b}) is equivalent to
\begin{eqnarray}\label{121123}
\begin{split}
&\frac{\be_{3;3;0}\theta^2\al(2\al-1)(\al-1)}{(\si^2+\be_{2;2;0}+\th)(6\si^2+6\be_{2;2;0}-4\be_{3;3;0}+3\th)}
=&0.
\end{split}
\end{eqnarray}
Note that \[6\si^2+6\be_{2;2;0}-4\be_{3;3;0}+3\th=6\si^2+6\be_{3;2;1}+2\be_{3;3;0}+3\th>0,\]
$\theta>0$ and $\alpha\neq 0,1/2,1$. Consequently, we have $\be_{3;3;0}=0$. By Lemma \ref{lem0928a} we have
$\Xi_{0}=0$. Hence, $X_t$ becomes the usual Fleming-Viot process with parent
independent mutation, which is known to be reversible. \qed
\end{proof}

In the following we  consider the reversibility for several classes of $(\Xi,A) $-Fleming-Viot processes under the condition $E=\{e_1,e_2\}$ with
$\nu_0(\{e_1\})=1/2=\nu_0(\{e_2\})$. The proofs will be given in the Appendix.

For $\beta\in (0,2)$, the $(\text{Beta}(2-\be,\be),A)$-Fleming-Viot
process is the $(\Xi,A)$-Fleming-Viot process with measure $\Xi$
supported on $[0,1]$ satisfying
\begin{eqnarray}\label{measure}
\Xi(dx)=\frac{\Gamma(2)}{\Gamma(2-\be)\Gamma(\be)}x^{1-\be}\(1-x\)^{\be-1}dx.
\end{eqnarray}

\begin{proposition}\label{beta}
Given any parent independent pure jump generator $A$ on
$E=\{e_1,e_2\}$ with $\nu_0(\{e_1\})=\nu_0(\{e_2\})=\al=1/2$, the
$(\text{Beta}(2-\be,\be),A)$-Fleming-Viot process is not reversible.
\end{proposition}

\begin{proposition}\label{alpha}
Given any parent independent pure jump generator $A$ on
$E=\{e_1,e_2\}$ with $\nu_0(\{e_1\})=\nu_0(\{e_2\})=\al=1/2$, the
$(\Xi,A)$-Fleming-Viot process with the finite measure $\Xi$ on
$[0,1]$ defined as
$$\Xi(dx)=x^{-\gamma}dx$$ for some $\gamma\in(0,1)$ is not
reversible.
\end{proposition}

\begin{proposition}\label{star}
Given any parent independent pure jump generator $A$ on
$E=\{e_1,e_2\}$ with $\nu_0(\{e_1\})=\nu_0(\{e_2\})=\al=1/2$, the
$(\Xi,A)$-Fleming-Viot process with $\Xi=\delta_1$, the point mass on
$[0,1]$, is not reversible.
\end{proposition}

Sagitov \cite{SS} introduced the Poisson-Dirichlet coalescent with
parameter $\epsilon$ as follows. Let $\Xi$ be the Poisson-Dirichlet
distribution $\Pi_{\epsilon}(dx)$ with a positive parameter
$\epsilon$ on the infinite simplex
\begin{equation*}
\Delta^{*}=\left\{\mathbf{x}\in\Delta:
\sum_{i=1}^{\infty}x_i=1\right\}.
\end{equation*}
Then the $\Pi_{\epsilon}$-coalescence rates are
\begin{eqnarray}\label{12320}
\la_{b;k_1,\ldots,k_r;s}=\frac{\epsilon^{r+s}}{\epsilon^{[b]}}\prod_{i=1}^r\(k_i-1\)!
\end{eqnarray}
for all $r\geq 1$, $k_1,\ldots,k_r\geq 2$ and $s\geq 0$, where
$b=k_1+\cdots+k_r+s$ and
$$\epsilon^{[b]}=\epsilon(\epsilon+1)\cdots\(\epsilon+b-1\)$$ is the ascending
factorial power.

\begin{proposition}\label{Poi-Dir}
Given a parent independent pure jump generator $A$ on
$E=\{e_1,e_2\}$ with $\nu_0(\{e_1\})=\nu_0(\{e_2\})=\al=1/2$, the
$(\Xi,A)$-Fleming-Viot process with $\Xi=\Pi_{\epsilon}$ is not
reversible.
\end{proposition}

\begin{remark}
For  $E=\{e_1,e_2\}$ with
$\nu_0(\{e_1\})=1/2=\nu_0(\{e_2\})$, we conjecture that the $\Xi$-Fleming-Viot process is reversible only if it degenerates to the classical Fleming-Viot process.
\end{remark}

\section{An SPDE for the $\left(\Xi, \frac{1}{2}\Delta\right)$-Fleming-Viot process}\label{sec5}

In this section we switch to a different topic and consider the weak and strong uniqueness for solution to an SPDE associated to the $\left(\Xi, \frac{1}{2}\Delta\right)$-Fleming-Viot process, where $\Delta=\frac{d^2}{dx^2}$ denotes the  Laplacian
mutation generator. Note that $\Delta$ is also used to represent the set of
infinite dimensional simplex in this paper.
Throughout this section we assume the type space $E$ to be the real
line $\bR$.

Next, we hope to characterize the $(\Xi,\frac{1}{2}\Delta)$-Fleming-Viot process by an SPDE
which has strong uniqueness. Here we adapt the approach of Xiong
\cite{X}.
For a measure $X$, we define the
distribution function
\begin{equation}\label{eq0915d}
u(x)=X((-\infty,x]),\qquad\forall x\in\RR.\end{equation}
Let $u^{-1}_{s-}(y)=\inf\{x:u_{s-}(x)\geq y\}$. Consider the
following SPDE
\begin{eqnarray}
u_t(x)&=&u_0(x)+\si\int^t_0\int^1_0\(1_{y\le
u_{s-}(x)}-u_{s-}(x)\)B(ds
dy)+\int^t_0\frac{1}{2}\Delta u_{s-}(x)ds\nonumber\\
&&+\int^t_0\int_\Delta\int_{[0,1]^\mathbb{N}}\(\sum_{i=1}^\infty z_i1_{y_i\le
u_{s-}(x)}-\sum_{i=1}^\infty z_i u_{s-}(x)\)M(ds dz dy),\label{eq0915e}
\end{eqnarray}
where $B(ds dy)$ is a white noise on $\RR_+\times(0,1]$ with
intensity $ds dy$ and $M(ds dz dy)$ is an independent Poisson random
measure on $\RR_+\times\Delta\times [0,1]^\mathbb{N}$ with intensity
$ds \(\sum_{i=1}^\infty z_i^2\)^{-1}\Xi_0(dz)
(dy)^{\otimes\mathbb{N}}$.

Denote $\< f,g\>:=\int_\RR f(x)g(x)dx$ for functions $f$ and $g$ on $\RR$.
We say an $\RR$-valued random field $u=\{u_t(x): t\geq0, x\in\RR \}$ is a solution to the SPDE
(\ref{eq0915e}) if for any $f\in C^2_0(\RR)$, the collection of compactly supported functions on
$\RR$ with continuous second derivatives,
\begin{equation*}
\begin{split}
\<u_t,f\>=& \<u_0,f\>+\si\int^t_0\int^1_0\int_\RR\(1_{y\le
u_{s-}(x)}-u_{s-}(x)\)f(x)dx B(dsdy)+\int^t_0\<u_{s-},\frac{1}{2}f''\>ds\\
&+\int^t_0\int_\Delta\int_{[0,1]^\mathbb{N}}\int_\RR\(\sum_{i=1}^\infty
z_i1_{y_i\le
u_{s-}(x)}-\sum_{i=1}^\infty z_i u_{s-}(x)\)f(x)dx M(ds dz dy).\\
\end{split}
\end{equation*}

A weak solution to the SPDE (\ref{eq0915e}) is a probability space
$(\Omega,\mathcal{F}, \mathbb{P})$ together with a filtration
$(\mathcal{F}_t)$ and adapted processes $(u,B,M)$ such that
(\ref{eq0915e}) holds.

\begin{theorem}
The SPDE (\ref{eq0915e}) has a weak solution. If the random field
$u_t(x)$ is a solution to the SPDE (\ref{eq0915e}), then the
measure-valued process $X$ defined by (\ref{eq0915d}) is a
$(\Xi, \frac{1}{2}\Delta)$-Fleming-Viot process.
Consequently, the SPDE (\ref{eq0915e}) has a unique weak solution.
\end{theorem}

\begin{proof}
Suppose that the SPDE (\ref{eq0915e}) has a solution.  For any $f\in
C^3_0(\RR)$, we have
\begin{eqnarray*}
\<X_t,f\>&=&-\<u_t,f'\>\\
&=&-\<u_0,f'\>-\si\int^t_0\int^1_0\int_\RR\(1_{y\le
u_{s-}(x)}-u_{s-}(x)\)f'(x)dx B(ds
dy)-\int^t_0\<u_{s-},\frac{1}{2} f'''\>ds\\
&&-\int^t_0\int_\Delta\int_{[0,1]^\mathbb{N}}\int_\RR\(\sum_{i=1}^\infty
z_i1_{y_i\le
u_{s-}(x)}-\sum_{i=1}^\infty z_i u_{s-}(x)\)f'(x)dx M(ds dz dy)\\
&=&\<X_0,f\>-\si\int^t_0\int^1_0\(\int^\infty_{u^{-1}_{s-}(y)}df(x)-\<u_{s-},f'\>\)
B(ds
dy)+\int^t_0\<X_{s-},\frac{1}{2}f''\>ds\\
&&-\int^t_0\int_\Delta\int_{[0,1]^\mathbb{N}}\(\sum_{i=1}^\infty
z_i\int^\infty_{u^{-1}_{s-}(y_i)}df(x)
-\sum_{i=1}^\infty z_i\<u_{s-},f'\>\) M(ds dz dy)\\
&=&\<X_0,f\>+\si\int^t_0\int^1_0\(f(u^{-1}_{s-}(y))-\< X_{s-},f\>\)
B(ds
dy)+\int^t_0\< X_{s-},\frac{1}{2} f''\>ds\\
&&+\int^t_0\int_\Delta\int_{[0,1]^\mathbb{N}}\(\sum_{i=1}^\infty
z_if(u^{-1}_{s-}(y_i))-\sum_{i=1}^\infty z_i\<X_{s-},f\>\) M(ds dz
dy).
\end{eqnarray*}

 Thus,
\[M_t(f)=\<X_t,f\>-\<X_0,f\>-\int^t_0\<X_{s-},\frac{1}{2}f''\>ds\]
is a square-integrable martingale with continuous part
\[M^c_t(f)=\si\int^t_0\int^1_0\(f(u^{-1}_{s-}(y))-\<X_{s-},f\>\)
B(ds dy)\] and pure jump part
\[M^d_t(f)=\int^t_0\int_\Delta\int_{[0,1]^\mathbb{N}}
\(\sum_{i=1}^\infty z_i f(u^{-1}_{s-}(y_i))-\sum_{i=1}^\infty
z_i\<X_{s-},f\>\) M(ds dz dy).\]
 Then
\begin{eqnarray*}
\<M^c(f)\>_t&=&\si^2\int^t_0\int^1_0\(f(u^{-1}_{s-}(y))-\<X_{s-},f\>\)^2dy
ds\\
&=&\si^2\int^t_0\int^1_0\(f(u^{-1}_{s}(y))^2-2\<X_{s},f\>f(u^{-1}_{s}(y))+\<X_{s},f\>^2\)dy
ds\\
&=&\si^2\int^t_0\(\<X_{s},f^2\>-\<X_{s},f\>^2\)ds.
\end{eqnarray*}
Further,
\[\sum_{0<s\le t}1_B(\De
M_s)-\int^t_0\int_\Delta\int_{[0,1]^\mathbb{N}}1_B\(\sum^\infty_{i=1}z_i\(\de_{u^{-1}_{s-}(y_i)}-X_{s-}\)\)
\(\sum_{i=1}^\infty z_i^2\)^{-1}\Xi_0(dz) dy ds \] is a martingale.
As
\begin{eqnarray*}
&&\int^t_0\int_\Delta\int_{[0,1]^\mathbb{N}}1_B\(\sum^\infty_{i=1}z_i\(\de_{u^{-1}_{s-}(y_i)}-X_{s-}\)\)
\(\sum_{i=1}^\infty z_i^2\)^{-1}\Xi_0(dz) dy ds\\
&=&\int^t_0\int_\Delta\int_{\mathbb{R}^\mathbb{N}}1_B\(\sum^\infty_{i=1}z_i\(\de_{x_i}-X_{s}\)\)\ga(ds
dz dx),
\end{eqnarray*}
we see that
\[\sum_{0<s\le t}1_B(\De
M_s)-\int^t_0\int_\Delta\int_{\mathbb{R}^\mathbb{N}}1_B\(\sum^\infty_{i=1}z_i\(\de_{x_i}-X_{s}\)\)\ga(ds
dz dx)\] is a martingale. Therefore, $X_t$ is the generalized Fleming-Viot
process.

The existence of weak solution follows from Theorem \ref{SPDE-weak}.
 \qed
\end{proof}

Next, we consider the backward version with $T$ fixed:
\[v_t(x)=u_{T-t}(x),\qquad \forall\ t\in[0,T],\;\;x\in\RR.\]
Note that $v_t$ is left-continuous with right limit. We also define
the backward random measure
\[\tilde{B}([0,s]\times A)=B([T-s,T]\times A),\qquad \forall\
s\in[0,T],\;A\in\cB([0,1]).\] The random measure $\tilde{M}$ is
defined similarly. Then, $\{v_t(x)\}$ satisfies the following
backward SPDE:
\begin{eqnarray}\label{eq0915f}
v_t(x)&=&u_0(x)+\sigma\int^T_t\int^1_0\(1_{y\le
v_{s+}(x)}-v_{s+}(x)\)\tilde{B}(\hat{d}s dy)+\int^T_t\frac{1}{2}\Delta
v_{s+}(x)ds\nonumber\\
&&+\int^T_t\int_\Delta\int_{[0,1]^\mathbb{N}} \(\sum_{i=1}^\infty
z_i 1_{y_i\le v_{s+}(x)}-\sum_{i=1}^\infty z_i
v_{s+}(x)\)\tilde{M}(\hat{d}s dz dy),
\end{eqnarray}
where $\hat{d}s$ denotes the backward It\^o's integral.

It is clear that if (\ref{eq0915f}) has a unique solution, so does
(\ref{eq0915e}), and vice versa. To prove the uniqueness of the
solution to (\ref{eq0915f}), we adapt the idea of Xiong \cite{X} by
relating it to a backward triply stochastic differential equation
(BTSDE). Denote
\[W^{t,x}_s=x+W_s-W_t,\qquad\forall\ t\le s\le T,\]
where $W_t$ is a one-dimensional standard Brownian motion.

 Fix $t$ and $x$, we consider the following BTSDE: For $t\le s\le
 T$,
 \begin{eqnarray}\label{eq0915g}
 Y^{t,x}_s&=&u_0(W^{t,x}_T)+\sigma\int^T_s\int^1_0\(1_{y\le
 Y^{t,x}_{r+}}-Y^{t,x}_{r+}\)\tilde{B}(\hat{d}r dy)\nonumber\\
&&+\int^T_s\int_\Delta\int_{[0,1]^\mathbb{N}} \(\sum_{i=1}^\infty
z_i 1_{y_i\le
 Y^{t,x}_{r+}}-\sum_{i=1}^\infty z_i Y^{t,x}_{r+}\)\tilde{M}(\hat{d}r dz dy)
 -\int^T_sZ^{t,x}_{r}dW_r.
 \end{eqnarray}
Let $\cG_t=\sigma(\cG^1_t,\cG^2_t)$ where $\cG^1_t$
 (non-decreasing) and $\cG^2_t$ (non-increasing) are independent sigma
 field families such that for any $t$, $\cF^W_t\subset \cG^1_t$ and
 $\cF^{B.M}_{t,T}\subset\cG^2_t$, where
 \begin{equation}\label{eq1220a}
\cF^{B.M}_{t,T}=\si\(\begin{array}{c}
B([r,T]\times A_1),\; M([r,T]\times A_2\times A_3),\\
r\in[t,T],\;A_1\in\cB([0,1]),\; A_2\in\cB(\De)\mbox{ and
}A_3\in\cB([0,1]^\mathbb{N}) \end{array}\).\end{equation} The
solution $(Y^{t,x}_s,Z^{t,x}_s)$ to the BTSDE (\ref{eq0915g}) is
$\cG_s$-adapted.

To establish the connection between the backward SPDE
(\ref{eq0915f}) and the BTSDE (\ref{eq0915g}), we need the following
extended It\^o's formula.

\begin{lemma}
Let $Y_t$ be a process given by
\[Y_t=Y_T+\int^T_t\int_{U_1}G_1(r,u)B(\hat{d}r du)+\int^T_t\int_{U_2} G_2(r,v)\tilde{M}(\hat{d}r dv)
 -\int^T_tZ_{r}dW_r,\]
where $B$ is a Gaussian white noise with intensity $\mu_1(du)ds$ and
$M$ is a Poisson random measure with intensity $\mu_2(dv)ds$.
Further, $G_i:\;[0, T]\times U_i\times\Om\to\RR$, $i=1,2,$ are
$\cG_t$ adapted, where $\cG_t$ is defined by (\ref{eq1220a}) with
obvious modification. Then, for any $f\in C^2(\RR)$, we have
 \begin{eqnarray}\label{eq0916a}
 f(Y_t)&=&f(Y_T)+\int^T_t\int_{U_1}f'(Y_s)G_1(s,u)B(\hat{d}r
 du)-\int^T_tf'(Y_s)Z_sdW_s\nonumber\\
 &&+\int^T_t\int_{U_2}\(f(Y_{s+}+G_2(s,v))-f(Y_{s+})\)\tilde{M}(\hat{d}r dv)\nonumber\\
&& +\int^T_t\int_{U_2}\(f(Y_{s}+G_2(s,v))-f(Y_{s})-f'(Y_{s})G_2(s,v)\)\mu_2(dv)ds\nonumber\\
 &&+\frac12\int^T_t\int_{U_1}f''(Y_s)G_1(s,u)^2\mu_1(du)ds-\frac12\int^T_tZ^2_sf''(Y_s)ds.
 \end{eqnarray}
 \end{lemma}

The proof of (\ref{eq0916a}) follows from that of the standard It\^o's
formula for semimartingale driven by Brownian motion and Poisson
random measure; see Chapter II of {\cite{IkWa89}}. Here, we omit the detail.

Since the second order moments of $X_t$ are the same as those for
the classical Fleming-Viot process, by the same arguments as in the
proof of Lemma 2.3 of Konno and Shiga \cite{KoSh88} we can show that
almost surely, $X_t$ has a density for almost all $t$.

Let $T_\de$ be the Brownian semigroup with $p_\delta$ the corresponding
heat kernel.

\begin{lemma}
Suppose that $X_0$ has a density $\partial_x u_0\in L^2(\RR,
e^{-|x|}dx)$. Then, there exists $h\in L^2([0,T]\times\RR,
e^{-2|x|}dx dt)$ such that
\[\lim_{\de\to 0+}\EE\int^T_0\int_\RR (T_\de X_t(x)-h_t(x))^2 e^{-2|x|}dx
dt=0.\]
\end{lemma}

\begin{proof}
For any $\delta, \delta'>0$, by the moment duality we have for
$\lambda=\Xi_0(\Delta)$,
\begin{equation*}
\begin{split}
&\EE\int^T_0 dt\int_\RR (T_\de X_t(x)-T_{\de'} X_t(x))^2 e^{-2|x|}dx\\
&=\int^T_0 dt\int_\RR \EE\left(\int_\RR p_\de(x-y) X_t(dy)-\int_\RR p_{\de'}(x-y) X_t(dy)\right)^2 e^{-2|x|}dx\\
&=\int^T_0dt\int_\RR e^{-2|x|}dx\int_0^t\la e^{-\la s}ds\int_\RR X_0(dy)\int_{\RR}p_{t-s}(y-z)\left(p_{s+\de}(x-z)-p_{s+\de'}(x-z)\right)^2dz\\
&\quad +\int^T_0 e^{-\lambda t}dt\int_\RR e^{-2|x|}dx\left(\int_\RR X_0(dy)p_{t+\de}(x-y)-\int_\RR X_0(dy) p_{t+\de'}(x-y)\right)^2.\\
\end{split}
\end{equation*}
For $X_0(dy)=X_0(y)dy$,
\begin{equation*}
\begin{split}
&\int_\RR e^{-2|x|}dx\left(\int_\RR (p_{t+\de}(x-y)- p_{t+\de'}(x-y))X_0(y)dy\right)^2\\
&\leq \int_\RR e^{-2|x|}dx\int_\RR \left(p_{t+\de}(x-y)-
p_{t+\de'}(x-y)\right)^2 e^{|y|}dy\int_{\RR} X_0(y)^2e^{-|y|}dy.
\end{split}
\end{equation*}
We can show that
\begin{equation}\label{1}
\begin{split}
&|p_{t+\de}(z)- p_{t+\de'}(z)| \leq t^{-2}|\de'-\de|
C(T)(1+z^2)(p_{t+\de}(z)+p_{t+\de'}(z)).
\end{split}
\end{equation}
Then for $0<\alpha<1/4$,
\begin{equation*}
\begin{split}
&\int_\RR e^{-2|x|}dx\int_\RR \left(p_{t+\de}(x-y)- p_{t+\de'}(x-y)\right)^{\alpha+2-\alpha} e^{|y|}dy\\
&\leq |\de'-\de|^{\alpha} t^{-2\alpha}C(T,\alpha)\int_\RR e^{-2|x|}dx\int_\RR (1+(x-y)^2)^\alpha\left(p_{t+\de}(x-y)^2+p_{t+\de'}(x-y)^2\right) e^{|y|}dy\\
&\leq |\de'-\de|^{\alpha} t^{-2\alpha-1/2}C(T,\alpha).
\end{split}
\end{equation*}
Consequently,
\[\lim_{\de,\de'\goto 0+}\int^T_0 e^{-\lambda t}dt\int_\RR e^{-2|x|}dx\left(\int_\RR X_0(dy)p_{t+\de}(x-y)-\int_\RR X_0(dy) p_{t+\de'}(x-y)\right)^2=0.\]
By Equation (\ref{1}), for $0<\al<1/4$, we have
\begin{equation*}
\begin{split}
&\int_\RR
X_0(dy)\int_{\RR}p_{t-s}(y-z)\left(p_{s+\de}(x-z)-p_{s+\de'}(x-z)\right)^2dz\\
&\leq{\({2\pi (t-s)}\)^{-1/2}}\int_\RR
X_0(dy)\int_{\RR}\left(p_{s+\de}(x-z)-p_{s+\de'}(x-z)\right)^{\al+2-\al}dz\\
&\leq\(2\pi
(t-s)\)^{-1/2}s^{-2\al}|\delta-\delta^{'}|^{\al}C\(T,\al\)
\int_\RR (1+(x-z)^2)^\alpha\left(p_{s+\de}(x-z)^2+p_{s+\de'}(x-z)^2\right)dz\\
&\leq|\delta-\delta^{'}|^{\al}C\(T,\al\)(t-s)^{-{1}/{2}}s^{-2\al-1/2}.
\end{split}
\end{equation*}

Therefore,
\[\lim_{\de,\de'\goto 0+}\int^T_0dt\int_\RR e^{-2|x|}dx\int_0^t\la e^{-\la s}ds\int_\RR X_0(dy)
\int_{\RR}p_{t-s}(y-z)\left(p_{s+\de}(x-z)-p_{s+\de'}(x-z)\right)^2dz=0.\]
Then
\[\lim_{\de, \de'\to 0+}\EE\int^T_0\int_\RR (T_\de X_t(x)-T_{\de'} X_t(x))^2 e^{-2|x|}dx
dt=0\]
and the desired result follows.
\qed
\end{proof}
Although for a.e. $s$, the function  $h_s(x)$ is defined for a.e.
$x\in\RR$, the quantity $h_s(W^{t,x}_s)$ is a well-defined random
variable when $s>t$ since the law of $W^{t,x}_s$ is absolutely
continuous.
\begin{theorem}
Suppose that $\partial_x u_0\in L^2(\RR,e^{-|x|}dx)$. If
$\{v_t(x)\}$ is a solution to (\ref{eq0915f}) such that $v_t(x)$ is
differentiable in $x$ and
\[\EE\int^T_0\int_\RR (\partial_x v_t(x))^2 e^{-2|x|}dx dt<\infty,\]
then $v_t(x)=Y^{t,x}_t$ a.s., where $\{Y^{t,x}_s\}$ is a solution to the
BTSDE (\ref{eq0915g}).
\end{theorem}

\noindent
{\bf Sketch of the proof} \, Since the main idea is the same as that of
Theorem 4.1 in \cite{X}, we only give a sketch.
 Applying $T_\de$ to both sides of (\ref{eq0915f}), we
get for $v^\de_s:=T_\de v_s$
\begin{eqnarray*}
v^\de_t(x)&=&u^\de_0(x)+\int^T_t\int_{[0,1]}\int_\RR
G_1(y,v_{s+}(l))p_\de(x-l)dl\tilde{B}(\hat{d}s dy)\\
&&+\int^T_t\int_{\Delta}\int_{[0,1]^\mathbb{N}}\int_\RR
G_2(y,z,v_{s+}(l))p_\de(x-l)dl\tilde{M}(\hat{d}s dzdy)
+\int^T_t\frac12\De v^\de_{s+}(x)ds,\end{eqnarray*} where $G_1$ and
$G_2$ are the obvious integrands of (\ref{eq0915f}).

Let $s=t_0<t_1<\cdots<t_n=T$ be a
partition of $[s,T]$. Writing $T_\de v_s(W^{t,x}_s)-T_\de
v_T(W^{t,x}_T)$ into telescopic sum, we have
\begin{eqnarray*}
&&v^\de_s(W^{t,x}_s)-T_\de u_0(W^{t,x}_T)\\
&=&\sum^{n-1}_{i=0}\(v^\de_{t_i}(W^{t,x}_{t_i})-v^\de_{t_i}(W^{t,x}_{t_{i+1}})\)
+\sum^{n-1}_{i=0}\(v^\de_{t_i}(W^{t,x}_{t_{i+1}})-v^\de_{t_{i+1}}(W^{t,x}_{t_{i+1}})\)\\
&=&-\sum^{n-1}_{i=0}\int^{t_{i+1}}_{t_i}\frac12\De
v^\de_{t_i}(W^{t,x}_r)dr
-\sum^{n-1}_{i=0}\int^{t_{i+1}}_{t_i}\nabla v^\de_{t_i}(W^{t,x}_r)dW_r\\
&&+\sum^{n-1}_{i=0}\int^{t_{i+1}}_{t_i}\int_{[0,1]}\int_\RR
p_\de(W^{t,x}_{t_{i+1}}-l)G_1(z,v_{r+}(l))dl\tilde{B}(\hat{d}r dz)\\
&&+\sum^{n-1}_{i=0}\int^{t_{i+1}}_{t_i}\frac12\De
v^\de_{r+}(W^{t,x}_{t_{i+1}})dr\\
&&+\sum^{n-1}_{i=0}\int^{t_{i+1}}_{t_i}\int_{\Delta}\int_{[0,1]^\mathbb{N}}\int_\RR
p_\de(W^{t,x}_{t_{i+1}}-l)G_2(y,z,v_{r+}(l))dl\tilde{M}(\hat{d}r
dzdy).
\end{eqnarray*}
Letting the mesh size decrease to zero and taking $\de\to 0$, we see
that
\[Y^{t,x}_s=v_s(W^{t,x}_s)\mbox{ and
}Z^{t,x}_s=\partial_xv_s(W^{t,x}_s),\qquad
a.e.\;\;(s,\om)\in(t,T]\times\Om,\] is a solution to the BTSDE
(\ref{eq0915g}). By the continuity of $Y^{t,x}_s$ in $s$ and the continuity of
$v_s(W^{t,x}_s)$ at $s=t$, taking $s\downarrow t$, we get the
conclusion of the theorem. \qed

Applying Yamada-Watanabe's argument, we can get the following result.
\begin{theorem}
The BTSDE (\ref{eq0915g}) has at most one solution.
\end{theorem}

\begin{proof} We drop the superscript $(t,x)$ in (\ref{eq0915g}) for simplicity. Suppose that
(\ref{eq0915g}) has two solutions $(Y^i_t,Z^i_t),\;\;i=1,\;2$. Let
$(a_k)$ be a sequence of decreasing positive numbers  defined recursively by
\[a_0=1\mbox{ and } \int^{a_{k-1}}_{a_k}z^{-1}dz=k,\qquad k\ge 1.\]
Let $\psi_k$ be non-negative  continuous functions supported in
$(a_k,a_{k-1})$ satisfying
\[\int^{a_{k-1}}_{a_k}\psi_k(z)dz=1\mbox{ and }\psi_k(z)\le
2(kz)^{-1},\qquad\forall\ z\in\RR.\] Let
\[\phi_k(z)=\int^{|z|}_0dy\int^y_0\psi_k(x)dx,\qquad\forall\ z\in\RR.\]
Then, $\phi_k(z)\to|z|$ and $|z|\phi''_k(z)\le 2k^{-1}$.

Denote $U_1=[0,1]$ and $U_2=\De\times[0,1]^\mathbb{N}$. Let $G_1$
and $G_2$ be the integrands for the backward stochastic integrals in
(\ref{eq0915g}). Note that
\begin{eqnarray*}
Y^1_t-Y^2_t&=&\int^T_t\int_{U_1}
\(G_1(u,Y^1_s)-G_1(u,Y^2_s)\)B(\hat{d}s
du)-\int^T_t\(Z^1_s-Z^2_s\)dW_s\nonumber\\
&&+\int^T_t\int_{U_2}
\(G_2(u,Y^1_s)-G_2(u,Y^2_s)\)\tilde{M}(\hat{d}s du).\end{eqnarray*}
By It\^o's formula, we get
\begin{eqnarray}\label{eq0902c}
&&\phi_k(Y^1_t-Y^2_t)\nonumber\\
&=&\int^T_t\int_{U_1}\phi'_k(Y^1_s-Y^2_s)\(G_1(u,Y^1_s)-G_1(u,Y^2_s)\)B(\hat{d}s
du)\nonumber\\
&&-\int^T_t\phi'_k(Y^1_s-Y^2_s)\(Z^1_s-Z^2_s\)dW_s\nonumber\\
&&+\int^T_t\int_{U_2}\(\phi_k(Y^1_{s+}-Y^2_{s+}+G_2(u,Y^1_{s+})-G_2(u,Y^2_{s+}))
-\phi_k(Y^1_{s+}-Y^2_{s+})\)\tilde{M}(\hat{d}s
du)\nonumber\\
&&+\frac12\int^T_t\int_{U_1}\phi''_k(Y^1_s-Y^2_s)\(G_1(u,Y^1_s)-G_1(u,Y^2_s)\)^2\mu_1(du)
ds\nonumber\\
&&+\int^T_t\int_{U_2}\bigg(\phi_k(Y^1_{s}-Y^2_{s}+G_2(u,Y^1_{s})-G_2(u,Y^2_{s}))\nonumber\\
&&\qquad -\phi_k(Y^1_{s}-Y^2_{s})-\phi_k'(Y^1_{s+}-Y^2_{s+})(G_2(u,Y^1_{s})-G_2(u,Y^2_{s}))\bigg)\mu_2(du)ds\nonumber\\
&&-\frac12\int^T_t\phi''_k(Y^1_s-Y^2_s)\(Z^1_s-Z^2_s\)^2ds.
\end{eqnarray}
Note that
\[\int_{U_i}|G_i(u,y_1)-G_i(u,y_2)|^2\mu_i(du)\le K|y_1-y_2|.\]
Taking expectation on both sides of (\ref{eq0902c}), we  get
\[
\EE\phi_k(Y^1_t-Y^2_t)\le
K\EE\int^T_t\phi''_k(Y^1_s-Y^2_s)|Y^1_s-Y^2_s|ds\le KTk^{-1}.
\]
Taking $k\to\infty$ and making use of Fatou's lemma, we have
\[\EE|Y^1_t-Y^2_t|\le 0.\]
Therefore, $Y^1_t=Y^2_t$ a.s. The uniqueness for $Z$ follows easily.
\qed
\end{proof}

As a consequence of the last two theorems, we see that the SPDE
(\ref{eq0915e}) has a unique strong solution.

\section{Appendix}
We consider the case of $E=\left\{e_1,e_2\right\}$ with
$\nu_0(e_1)=\nu_0(e_2)=1/2$ in this Appendix. Let $F=\left\{e_1\right\}$.

For convenience, we denote
$a_2=\be_{2;2;0}+\si^2$, $a_i=\be_{i;i;0},\;i=3,4,5,6$,
$a_{21}=\be_{3;2;1}+\sigma^2$, $a_{211}=\be_{4;2;2}+\si^2$,
$a_{22}=\be_{4;2,2;0}$, $a_{31}=\be_{4;3;1}$,
$a_{2111}=\be_{5;2;3}+\si^2$, $a_{221}=\be_{5;2,2;1}$,
$a_{311}=\be_{5;3;2}$, $a_{32}=\be_{5;3,2;0}$, $a_{41}=\be_{5;4;1}$,
$a_{21111}=\be_{6;2;4}+\si^2$, $a_{2211}=\be_{6;2,2;2}$,
$a_{3111}=\be_{6;3;3}$, $a_{222}=\be_{6;2,2,2;0}$,
$a_{321}=\be_{6;3,2;1}$, $a_{411}=\be_{6;4;2}$,
$a_{33}=\be_{6;3,3;0}$, $a_{42}=\be_{6;4,2;0}$,
$a_{51}=\be_{6;5;1}$. Let
$$\mathbf{m}_i=\int_{M_1(E)}\mu^{i}(\left\{e_1\right\})\Pi(d\mu),$$
where $i=1,2,\ldots,6$ and $\Pi$ is the invariant measure. By the
consistency condition, we have
\begin{eqnarray}\label{con26}
\begin{cases}
a_{21}=a_2-a_3,\\
a_{22}=a_2-2a_3+a_4-a_{211},\\
a_{31}=a_3-a_4,\\
a_{41}=a_4-a_5,\\
a_{221}=\frac15a_{{2}}-\frac45a_{{3}}+a_{{4}}+\frac15a_{{211}}-\frac25a_{{2111}}-\frac25
a_{{5}},\\
a_{311}=\frac35a_{{211}}-\frac25a_{{2}}+\frac85a_{{3}}-2a_{{4}}-\frac15a_{{2111}}+\frac4
5a_{{5}},\\
a_{32}=-\frac35a_{{3}}+\frac15a_{{2111}}+\frac15a_{{5}}-\frac35a_{{211}}+\frac25a_{{2}},\\
a_{2211}=-\frac
{14}{15}a_{{3}}+a_{{4}}+\frac25a_{{2}}-\frac45a_{{5}}-\frac23a_{{
33}}-a_{{42}}-\frac35a_{{211}}+\frac {8}{15}a_{{2111}}-\frac13a_{{
21111}}+\frac13a_{{6}},\\
a_{222}=2\,a_{{211}}-a_{{2}}+\frac43\,a_{{3}}+\frac83\,a_{{33}}+3\,a_{{42}}-\frac43\,a_{{
2111}}+\frac13\,a_{{21111}}-\frac13\,a_{{6}},\\
a_{3111}=-\frac35\,a_{{2111}}+{\frac
{14}{5}}\,a_{{3}}-3\,a_{{4}}-\frac65\,a_{{2}}+{ \frac
{12}{5}}\,a_{{5}}+2\,a_{{33}}+3\,a_{{42}}+\frac95\,a_{{211}}-a_{{6}},\\
a_{321}=\frac25\,a_{{2}}-\frac35\,a_{{3}}+\frac15\,a_{{5}}-a_{{33}}-a_{{42}}-\frac35\,a_{{211}
}+\frac15\,a_{{2111}},\\
a_{411}=a_{{4}}-2\,a_{{5}}-a_{{42}}+a_{{6}},\\
a_{51}=a_5-a_6.\\
\end{cases}
\end{eqnarray}

Taking $p=1$, $q=0$ and $f=1_F$ in Equation (\ref{eq0915c}), we get
\begin{equation}\label{11231}
\mathbf{m}_1=\al=1/2.
\end{equation}

Taking $p=2$, $q=0$ and $f=1_{F\times F}$ in Equation
(\ref{eq0915c}), we get
\begin{equation}\label{11232}
\mathbf{m}_2
={\frac { a_{{2}}+\theta\alpha}{a_{{2}}+\theta} }
\mathbf{m}_1
=\frac{2a_2+\theta}{4\(a_2+\theta\)}.
\end{equation}

Taking $p=3$, $q=0$ and $f=1_{F\times F\times F}$ in Equation
(\ref{eq0915c}), we get
\begin{equation}\label{11233}
\begin{split}
\mathbf{m}_3
=&{\frac {a_{{3}}\mathbf{m}_1+ \left(
3\,a_{{21}}+\frac32\,\theta\,\alpha
\right)\mathbf{m}_2}{3\,a_{{21}}+a_{{3}}+\frac32\,\theta}}
=\frac{4a_2+\theta}{8\(a_2+\theta\)},
\end{split}
\end{equation}
 where we have applied the consistency condition $a_{21}=a_2-a_3$ and canceled the common positive factor $$3a_2-2a_3+\frac32\theta=3a_{21}+a_3+\frac32\theta$$ from both the denominator and the numerator.

Taking $p=4$, $q=0$ and $f=1_{F\times F\times
F\times F}$ in Equation (\ref{eq0915c}), we get
\begin{equation}\label{11234}
\begin{split}
\mathbf{m}_4
=&{\frac {a_{{4}}\mathbf{m}_1+ \left(
4\,a_{{31}}+3\,a_{{22}} \right)\mathbf{m}_2 + \left(
6\,a_{{211}}+2\,\theta\,\alpha \right)\mathbf{m}_3}{6\,a_{{211}}+
3\,a_{{22}}+4\,a_{{31}}+a_{{4}}+2\,\theta}}\\
=&\frac{\theta^2-4a_3\theta+10a_2\theta+2a_4\theta+12a_{211}a_2+12a_2^2-8a_3a_2}{8\(3a_{211}+3a_2-2a_3+2\theta\)\(a_2+\theta\)}\\
=&\frac{\theta^2-4a_3\theta+10a_2\theta+2a_4\theta+12a_{211}a_2+12a_2^2-8a_3a_2}{8\(6\,a_{{211}}+
3\,a_{{22}}+4\,a_{{31}}+a_{{4}}+2\,\theta\)\(a_2+\theta\)},
\end{split}\end{equation}
where $3a_{211}+3a_2-2a_3+2\theta=6\,a_{{211}}+
3\,a_{{22}}+4\,a_{{31}}+a_{{4}}+2\,\theta$ follows from the consistency condition.\\

Taking
$p=1$, $q=3$, $f=1_F$ and $g=1_{F\times F\times F}$ in Equation
(\ref{eq0915c}), we get
\begin{equation}\label{121123e2}
\begin{split}
\(3a_{{21}}+\theta\,\alpha\)\mathbf{m}_3+a_{{3}}\mathbf{m}_2-\(3a_{21}+a_{{3}}+\th\)\mathbf{m}_4=0.
\end{split}\end{equation}
Plugging in Equations (\ref{11232})-(\ref{11234}), Equation (\ref{121123e2}) becomes
\begin{equation*}
\begin{split}
\frac{\theta B}{\(3a_{211}+3a_2-2a_3+2\theta\)\(a_2+\theta\)}
&=\frac{\theta B}{\(6\,a_{{211}}+3\,a_{{22}}+4\,a_{{31}}+a_{{4}}+2\,\theta\)\(a_2+\theta\)} \\
&=0,
\end{split}
\end{equation*}
where
\[B:=\(-6a_3+3a_2-3a_{211}+4a_4\)\theta+12a_4a_2-22a_3a_2+6a_3a_{211}-6a_{211}a_2-8a_4a_3+12a_3^2+6a_2^2.\]
By the consistency condition, we have
$$-6a_3+3a_2-3a_{211}+4a_4=a_4+3a_{22}.$$
If $a_4+3a_{22}=0$, then the $\Xi$-coalescent
 degenerates to Kingman's coalescent  by Lemma \ref{lem0928a};
otherwise, $a_4+3a_{22}>0$ and we have
\begin{eqnarray}\label{theta}
\begin{split}
\theta
=&{\frac
{2\(6\,a_{{4}}a_{{2}}-11\,a_{{3}}a_{{2}}+3\,a_{{3}}a_{{211}}-3
\,a_{{211}}a_{{2}}-4\,a_{{4}}a_{{3}}+6\,{a_{{3}}}^{2}+3\,{a_{{2}}}^{2}
\)}{-a_4-3a_{22}}},
\end{split}
\end{eqnarray}
which further imposes a necessary condition for the $(\Xi, A)$-Fleming-Viot process in case of $E=\left\{e_1,e_2\right\}$ with $\nu_0(\{e_1\})=\nu_0(\{e_2\})=1/2$.

With Equation (\ref{theta}) we can compute higher moments  to reach a contradiction.
Taking $p=5$, $q=0$ and $f=1_{F\times F\times F\times F\times F}$ in
Equation (\ref{eq0915c}), we get
\begin{equation}\label{11235}
\begin{split}
\mathbf{m}_5&=\frac {a_{{5}}\mathbf{m}_1+ \left(
10\,a_{{32}}+5\,a_{{41}} \right)\mathbf{m}_2+ \left(
15\,a_{{221}}+10\,a_{{311}} \right)\mathbf{m}_3+ \left( 10\,a_{{
2111}}+\frac52\,\theta\,\alpha
\right)\mathbf{m}_4}{10\,a_{{2111}}+15\,a_{{221
}}+10\,a_{{311}}+10\,a_{{32}}+5\,a_{{41}}+a_{{5}}+\frac52\,\theta}\\
&=\frac{\(\theta-6a_{211}+10a_4+20a_2-16a_3\)\theta-16a_3a_2+24a_2^2+24a_{211}a_2}{16\(a_2+\theta\)\(3a_{211}+3a_2-2a_3+2\theta\)}\\
&=\frac{\(\theta-6a_{211}+10a_4+20a_2-16a_3\)\theta-16a_3a_2+24a_2^2+24a_{211}a_2}{16\(a_2+\theta\)\(6\,a_{{211}}+
3\,a_{{22}}+4\,a_{{31}}+a_{{4}}+2\,\theta\)},
\end{split}\end{equation}
where we have applied the consistency condition and canceled a common positive factor
$$4a_{2111}+3a_2-2a_3+3a_{211}+\frac52\theta
=10a_{2111}+15a_{221}+10a_{311}+10a_{32}+5a_{41}+a_5+\frac52\theta$$
from both the denominator and numerator.

Taking $p=6$, $q=0$ and $f=1_{F\times F\times F\times F\times
F\times F}$ in Equation (\ref{eq0915c}), we get
\begin{equation}\label{11236}
\begin{split}
\mathbf{m}_6=&\frac {a_{{6}}\mathbf{m}_1+ \left(
6\,a_{{51}}+15\,a_{{42}}+10\,a_{{33}}
 \right) \mathbf{m}_2+\left( 15\,a_{{411}}+60\,a_{{321}}+15\,a_{{222}}
 \right) \mathbf{m}_3}{C}\\
 &+\frac{\left( 20\,a_{{3111}}+45\,a_{{2211}} \right)\mathbf{m}_4 +
\left( 15\,a_{{21111}}+3\,\theta\,\alpha \right)
\mathbf{m}_5}{C},
\end{split}
\end{equation}
where
\[C:=a_{{6}}+6
\,a_{{51}}+15\,a_{{42}}+10\,a_{{33}}+15\,a_{{411}}+60\,a_{{321}}+15\,a
_{{222}}+20\,a_{{3111}}+45\,a_{{2211}}+15\,a_{{21111}}+3\,\theta.\]

Taking $p=1$, $q=5$, $f=1_F$ and $g=1_{F\times F\times F\times
F\times F}$ in Equation (\ref{eq0915c}), we get
\begin{equation}\label{121123e5}
\begin{split}
&a_{{5}}\mathbf{m}_2+ \left( 10\,a_{{32}}+5\,a_{{41}} \right)
\mathbf{m}_3+
 \left( 10\,a_{{311}}+15\,a_{{221}} \right) \mathbf{m}_4+ \left( 10\,a_{{
2111}}+2\,\theta\,\alpha \right) \mathbf{m}_5\\
&- \left( a_{{5}}+10\,a_{{32}}+
5\,a_{{41}}+10\,a_{{311}}+15\,a_{{221}}+10\,a_{{2111}}+2\,\theta
 \right)\mathbf{m}_6=0.
\end{split}\end{equation}
Plugging in Equations (\ref{11232}), (\ref{11233}), (\ref{11234}), (\ref{11235}), (\ref{11236}) and (\ref{theta}), Equation (\ref{121123e5})  only involves the collision rates. We can then consider examples of coalescents in which those rates are specified.

\noindent
{\bf Proof of Proposition \ref{beta}}
Let $\mathbb{M}_n=\int_0^1x^n\Xi\(dx\)$. It follows that
\begin{equation*}
\mathbb{M}_n=\frac{\Gamma\(n+2-\be\)}{\(n+1\)!\Gamma\(2-\be\)}.
\end{equation*}
By integral representation  of the coalescence rates, we have
$$a_2=\mathbb{M}_0, a_3=\mathbb{M}_1, a_4=\mathbb{M}_2, a_{211}=\mathbb{M}_0-2\mathbb{M}_1+\mathbb{M}_2,$$
$$a_5=\mathbb{M}_3, a_{2111}=\mathbb{M}_0-3\mathbb{M}_1+3\mathbb{M}_2-\mathbb{M}_3, a_{42}=0,$$ $$a_{33}=0,
a_6=\mathbb{M}_4,
a_{21111}=\mathbb{M}_0-4\mathbb{M}_1+6\mathbb{M}_2-4\mathbb{M}_3+\mathbb{M}_4.$$
The value of $\theta$ can be obtained by Equation (\ref{theta}).
Plugging in the values of coalescence rates and $\theta$ to Equation
(\ref{121123e5}), we have
\begin{eqnarray*}
 \frac{\left( \beta-2 \right)  \left( \beta-3
 \right)  \left( \beta+1 \right)  \left( {\beta}^{4}+8\,{\beta}^{3}-39
\,{\beta}^{2}+6\,\beta+72 \right) {\beta}^{2}}
{\(\be^2+3\)\(\be^4+6\be^3-\be^2-126\be-72\)}=0.\\
\end{eqnarray*}
We can not find any $\be\in(0,2)$ to satisfy the above equation,
which contradicts the reversibility. Therefore, the
$(\text{Beta}(2-\be,\be),A)$-Fleming-Viot process is not reversible.
\qed

\noindent
{\bf Proof of Proposition \ref{alpha}}  Let
$\mathbb{M}_n=\int_0^1x^n\Xi\(dx\)$. We have
\begin{equation*}
\mathbb{M}_n=1/\(n+1-\gamma\).
\end{equation*}
The values of $a_2$, $a_3$, $a_4$, $a_{211}$, $a_5$, $a_{2111}$,
$a_6$, $a_{21111}$ and $\theta$ can be similarly expressed in terms
of $\gamma$. Note that $a_{42}=a_{33}=0$. With the consistency
condition (\ref{con26}) and the moments obtained from Equations
(\ref{11231})-(\ref{11234}), (\ref{11235})-(\ref{11236}), we plug in those values to Equation
(\ref{121123e5}). It follows that
\begin{eqnarray*}
{\frac {-1+\gamma}{ \left( -4+\gamma \right)  \left( -6+\gamma
 \right)  \left( {\gamma}^{3}-14\,{\gamma}^{2}+61\,\gamma-120 \right)
}}=0.
\end{eqnarray*}
There is no $\gamma\in(0,1)$ satisfying the above equation.
Therefore, such a $(\Xi,A)$-Fleming-Viot process  is not reversible.
\qed

\noindent
{\bf Proof of Proposition \ref{star}}
If $\Xi=\delta_1$, the corresponding coalescent  only allows
all the blocks to merge into a single block. Then we have $a_2=1$,
$a_3=1$, $a_{21}=0$, $a_4=1$, $a_{31}=0$, $a_{22}=0$ and
$a_{211}=0$. By Equations (\ref{11232})-(\ref{11234}), the moments
 $\mathbf{m}_2$, $\mathbf{m}_3$ and $\mathbf{m}_4$ can be expressed
by $\theta$. Plugging in these values to Equation (\ref{121123e2}),
we have
\begin{equation*}
\frac{\theta^2}{\(1+\theta\)\(1+2\theta\)}=0.
\end{equation*}
Since $\theta$ is positive, we get a contradiction. Therefore, the
corresponding $\(\Xi,A\)$-Fleming-Viot process is not reversible.
\qed

\noindent
{\bf Proof of Proposition \ref{Poi-Dir}}
By Equation (\ref{12320}), the coalescence rates $a_2$, $a_3$,
$a_{21}$, $a_{4}$, $a_{31}$, $a_{22}$, $a_{211}$, $a_5$, $a_{41}$,
$a_{32}$, $a_{311}$, $a_{221}$, $a_{2111}$, $a_6$, $a_{51}$,
$a_{42}$, $a_{33}$, $a_{411}$, $a_{321}$, $a_{222}$, $a_{3111}$,
$a_{2211}$ and $a_{21111}$ are all available. Then the moments
$\mathbf{m}_1$, $\mathbf{m}_2$, $\mathbf{m}_3$, $\mathbf{m}_4$,
$\mathbf{m}_5$ and $\mathbf{m}_6$ can be expressed by $\epsilon$ and
$\theta$. The value of $\theta$ can be obtained by Equation
(\ref{theta}). Replacing all of these values in Equation
(\ref{121123e5}), it follows that
\begin{equation*}
\frac{\(10\epsilon^2+11\epsilon+6\)\epsilon^2}
{\(5\epsilon+6\)\(6+11\epsilon\)\(17\epsilon^4+109\epsilon^3+319\epsilon^2+394\epsilon+120\)\(1+\epsilon\)}=0.
\end{equation*}
We can not find any positive $\epsilon$ to satisfy the above
equation. Therefore, the corresponding $\(\Xi,A\)$-Fleming-Viot
process is not reversible. \qed

\end{document}